\documentclass{amsart}

\usepackage{amssymb}
\usepackage[utf8]{inputenc}
\usepackage[T1]{fontenc}
\usepackage[english]{babel}
\usepackage{xargs}
\usepackage{graphicx, color}
\usepackage[figurewithin=none]{caption,subcaption} %option for change numering of figures
\usepackage{enumitem}
\usepackage{hyperref}
\hypersetup{
	hidelinks,
  colorlinks  = true,    % Colours links instead of ugly boxes
  urlcolor    = blue,    % Colour for external hyperlinks
  linkcolor   = blue,    % Colour of internal links
  citecolor   = blue,      % Colour of citations
  pdfauthor   = {Antonio Montero},%
  pdfsubject  = {Chiral polytopes},%
  pdftitle    = {On the Schlafli symbol of chiral extensions of polytopes},  %
}

\usepackage[capitalise,noabbrev, nameinlink]{cleveref}
\Crefname{step}{Step}{Steps}

\theoremstyle{plain}
\newtheorem{thm}{Theorem}[section]

\newtheorem{lem}[thm]{Lemma}

\newtheorem{prop}[thm]{Proposition}

\newtheorem{coro}[thm]{Corollary}

\theoremstyle{definition}

\newtheorem{exam}[thm]{Example}

\theoremstyle{remark}

\newtheorem{rem}[thm]{Remark}

\numberwithin{equation}{section}

\renewcommand{\leq}{\leqslant} %
\renewcommand{\geq}{\geqslant}
\renewcommand{\epsilon}{\varepsilon} %
\renewcommand{\subset}{\subseteq} %

\renewcommand{\{}{\lbrace}
\renewcommand{\}}{\rbrace}
\renewcommand{\bar}{\overline}
\DeclareMathOperator{\rk}{rk}
\DeclareMathOperator{\cay}{Cay}
\DeclareMathOperator{\GPR}{\mathcal{G}}
\DeclareMathOperator{\fl}{\mathcal{F}}
\DeclareMathOperator{\fw}{\mathcal{F}^{w}}
\DeclareMathOperator{\con}{Con}
\DeclareMathOperator{\conp}{\con^{+}}
\DeclareMathOperator{\conw}{\con^{w}}

\DeclareMathOperator{\aut}{Aut}
\DeclareMathOperator{\autp}{\aut^{+}}
\DeclareMathOperator{\lcm}{lcm}
\DeclareMathOperator{\stab}{Stab}
\newcommand{\vect}[1]{\bar{\mathrm{#1}}} 
\newcommand{\cP}{\mathcal{P}}
\newcommand{\cQ}{\mathcal{Q}}
\newcommand{\cK}{\mathcal{K}}
\newcommand{\cM}{\mathcal{M}}
\newcommand{\cN}{\mathcal{N}}
\newcommand{\cR}{\mathcal{R}}
\newcommand{\cG}{\mathcal{G}}
\newcommand{\bN}{\mathbb{N}}
\newcommand{\bZ}{\mathbb{Z}}
\newcommand{\vx}{\vect{x}}
\newcommand{\vy}{\vect{y}}

\newcommand{\va}{\vect{a}}
\newcommand{\ve}{\vect{e}}

\newcommand{\bLL}{\Lambda}

\newcommandx{\id}{\epsilon}
\newcommandx{\twoSM}[2][1=\cM, 2=s] {\hat{2}#2^{#1 - 1}}
\newcommandx{\gpr}[2][1=\cK, 2=s]{\GPR_{#2}(#1)}
\newcommand{\mix}{\diamondsuit}

\newcommand{\baseFlag}{\Phi_{0}}
\newcommandx{\Proot}[2][1=\cP, 2= \baseFlag, usedefault]{\left( #1,#2 \right)}
\begin{document}

%%
%% The title of the paper goes here.  Edit to your title.
%%

\title{On the Schläfli symbol of chiral extensions of polytopes}

%%
%% Now edit the following to give your name and address:
%% 

\author{Antonio Montero}
\address{Department of Mathematics and Statistics, York University, Toronto, Ontario M3J 1P3, Canada}
\email{amontero@yorku.ca}
% \urladdr{} % Delete if not wanted.

%%%
%%% The following is for the abstract.  The abstract is optional and
%%% if not used just delete, or comment out, the following.
%%%



\begin{abstract}
Given an abstract $n$-polytope $\cK$, an abstract $(n+1)$-polytope $\cP$ is an extension of $\cK$ if all the facets of $\cP$ are isomorphic to $\cK$. A chiral polytope is a polytope with maximal rotational symmetry that does not admit any reflections. If $\cP$ is a chiral extension of $\cK$, then all but the last entry of the Schläfli symbol of $\cP$ are determined. In this paper we introduce some constructions of chiral extensions $\cP$ of certain chiral polytopes in such a way that the last entry of the Schläfli symbol of $\cP$ is arbitrarily large.
\end{abstract}

%%
%%  LaTeX will not make the title for the paper unless told to do so.
%%  This is done by uncommenting the following.
%%

\maketitle

%% To print the list of todos
%\listoftodos\relax 

\section{Introduction} \label{sec:intro}
Highly symmetric polyhedra have been of  interest to humanity not only for their mathematical structure but also for their degree of beauty. 
There exists evidence that the five Platonic Solids were known before the Greeks. However, undoubtedly the Greeks have the credit of collecting and formalising the mathematical knowledge about these objects. 
The formal study of highly symmetric polyhedra-like structures took a new breath in the $20^{th}$ century. 
The work by Coxeter and Grünbaum was significant and set the basis for what today we know as abstract polytopes.

Abstract polytopes are combinatorial structures that generalise (the face-lattice of) convex polytopes as well as the geometric polytopes considered by Coxeter and Grünbaum. 
They also include face-to-face tilings of Euclidean and Hyperbolic spaces as well as most maps on surfaces.
Abstract polytopes were introduced by Danzer and Schulte in \cite{DanzerSchulte_1982_RegulareInzidenzkomplexe.I} where they describe the basic properties of these objects.
The early research on abstract polytopes was focused on the so called \emph{abstract regular polytopes}. These are the abstract polytopes that admit a maximal degree of symmetry by combinatorial reflections. 
Much of the rich theory of abstract regular polytopes can be found in \cite{McMullenSchulte_2002_AbstractRegularPolytopes}.

Abstract polytopes inherit a natural recursive structure from their convex and geometric analogues: a cube can be thought as a family of six squares glued along their edges. 
In general, an $n$-polytope $\cP$ can be thought as a family of $(n-1)$-polytopes glued along $(n-2)$-faces.
These $(n-1)$-polytopes are the \emph{facets} of $\cP$ and whenever all these polytopes are \emph{isomorphic} to a fixed polytope $\cK$ we say that $\cP$ is an extension of $\cK$.

The local combinatorics of a highly symmetric $n$-polytope is described by its Schläfli symbol, whenever it is well-defined. 
In particular, the last entry of this symbol describes how many facets are incident to each $(n-3)$-face.
If $\cK$ has Schläfli symbol $\left\{ p_{1}, \dots, p_{n-2} \right\} $ and $\cP$ is an extension of $\cK$ with a well-defined Schläfli symbol, then it must be $\left\{ p_{1}, \dots, p_{n-2}, p_{n-1} \right\} $ for some $p_{n-1} \in \left\{ 2, \dots, \infty \right\} $.

The problem of determining whether or not a fixed polytope $\cK$ admits an extension has been part of the theory's development since its beginning. In fact in \cite{DanzerSchulte_1982_RegulareInzidenzkomplexe.I} Danzer and Schulte attack this problem for regular polytopes.
They prove that every non-degenerate regular polytope $\cK$ admits an extension $\cP$ and this extension is finite if and only if $\cK$ is finite. Moreover, the last entry of the Schläfli symbol of $\cP$ is $6$.
In \cite{Danzer_1984_RegularIncidenceComplexes} Danzer proves that every non-degenerate polytope $\cK$ admits an extension and this extension is finite (resp. regular) if and only if $\cK$ is finite (resp. regular). If $\cK$ is regular, the last entry of the Schläfli symbol of the extension is $4$. In \cite{Pellicer_2009_ExtensionsRegularPolytopes, Pellicer_2010_ExtensionsDuallyBipartite} Pellicer develops several constructions that have as a consequence that every regular polytope admits a regular extension with prescribed even number as the last entry of the Schläfli symbol. 
On the other hand, in \cite{Hartley_2005_LocallyProjectivePolytopes}, Hartley proves that the $n$-hemicube, the polytope obtained by identifying oposite faces of a $n$-cube, cannot be extended with an odd number as the last entry of the Schläfli symbol.

Besides regular polytopes, another class of symmetric abstract polytopes that has been of interest is that of \emph{chiral polytopes}.
Chiral polytopes are those polytopes that admit maximal symmetry by abstract rotations but do not admit mirror reflections.

The problem of building chiral polytopes on higher ranks has proved to be fairly difficult. 
This problem has been attacked from several approaches, see for example 
\cite{ColbournWeiss_1984_CensusRegular$3$,NostrandSchulte_1995_ChiralPolytopesHyperbolic,SchulteWeiss_1994_ChiralityProjectiveLinear} for rank $4$ and \cite{ConderHubardPisanski_2008_ConstructionsChiralPolytopes,DAzevedoJonesSchulte_2011_ConstructionsChiralPolytopes} for rank $5$ and $6$. It was not until 2010 that Pellicer showed in \cite{Pellicer_2010_ConstructionHigherRank} the existence of chiral polytopes of all ranks higher than $3$.

The problem of finding chiral extensions of abstract polytopes has been one of the main approaches to find new examples of chiral polytopes. 
However, the results are less numerous than those concerning regular polytopes. 
If $\cP$ is a chiral $n$-polytope, then its facets are either orientably regular or chiral. 
In any case, the $(n-2)$-faces of $\cP$ must be regular (see \cite[Proposition 9]{SchulteWeiss_1991_ChiralPolytopes}). 
This implies that if $\cP$ is a chiral extension of $\cK$, then $\cK$ is either regular or chiral with regular facets. 

In \cite{CunninghamPellicer_2014_ChiralExtensionsChiral} Cunningham and Pellicer proved that any finite chiral polytope with regular facets admits a finite chiral extension. 
Their construction offers little control on the Schläfli symbol of the resulting extension.

In this work we introduce two constructions of extensions of chiral polytopes with regular facets that satisfy certain conditions. 
These constructions offer some control on the Schläfli symbol of the resulting extension. 
In particular we prove that under certain conditions, some chiral polytopes with regular facets admit chiral extensions whose Schläfli symbol has arbitrarily large last entry.

The paper organized as follows. 
In \cref{sec:basics} we introduce the basic concepts on highly symmetric polytopes and related topics needed to develop our results. 
In \cref{sec:dually} we introduce a construction of a GPR-graph of a chiral extension for a dually-bipartite chiral polytope. 
In \cref{sec:quotients} we define the maniplex $\twoSM$ as a generalisation of the polytope $2s^{\cK-1}$ introduced by Pellicer in \cite{Pellicer_2009_ExtensionsRegularPolytopes} and use it to build chiral extensions of chiral polytopes admitting some particular regular quotients. 
Finally, in \cref{sec:examples} we show explicit examples of how to use our construction for chiral maps on the torus.

\section{Basic notions} \label{sec:basics}
In this section we introduce the basic concepts on the theory of highly symmetric abstract polytopes. 
Our main references are \cite{McMullenSchulte_2002_AbstractRegularPolytopes, MonsonPellicerWilliams_2014_MixingMonodromyAbstract, SchulteWeiss_1991_ChiralPolytopes, Wilson_2012_ManiplexesPart1}.
%%%%%%%%%%%%%%%%%%%%%%%%%%%%%%%
%HightlySymmetricPolytopes%%%%%
%%%%%%%%%%%%%%%%%%%%%%%%%%%%%%%
\subsection{Regular abstract polytopes.}\label{sec:HSAP}
Abstract polytopes are combinatorial structures with a very geometrical nature.
They generalise convex polytopes, tilings of the Euclidean spaces, maps on surfaces, among others.
Formally, \emph{abstract polytope of rank $n$} or an \emph{$n$-polytope}, for short, is a partially ordered set $(\cP, \leq)$ (we usually omit the order symbol) that satisfies the properties in \cref{item:maxmin,item:flags,item:diamond,item:sfc} below.
\begin{enumerate}
 \item \label{item:maxmin} $\cP$ has a minimum element $F_{-1}$ and a maximum element $F_{n}$.
\end{enumerate}
The elements of $\cP$ are called \emph{faces}. 
We say that two faces $F$ an $G$ are \emph{incident} if $F \leq G$ or $G\leq F$.
A \emph{flag} is a maximal chain of $\cP$.
We require that
\begin{enumerate}[resume]
 \item \label{item:flags} Ever flag of $\cP$ contains exactly $n+2$ elements, including the maximum and the minimum of $\cP$.
\end{enumerate}
The condition in \cref{item:flags} allows us to define a \emph{rank function} $\rk: \cP \to \left\{ -1, \dots, n \right\} $ such that $\rk(F_{-1}) = -1$, $\rk(F_{n})=n$ and for every other face $F$, $\rk(F) = r$ if there exists a flag $\Phi$ of $\cP$ such that there are precisely $r$ faces $G$ in $\Phi$ satisfying $F_{-1} < G < F$.
Note that this definition does not depend on the choice of $\Phi$.
The rank function is a combinatorial analogue of the notion of dimension for convex polytopes and tillings of the space.
We usually call \emph{vertices}, \emph{edges} and \emph{facets} the elements of rank $0$, $1$ and $n-1$ respectively.
In general, a face of rank $i$ is called an $i$-face
We also need that $\cP$ satisfies the \emph{diamond condition} described in \cref{item:diamond}. 
\begin{enumerate}[resume]
 \item \label{item:diamond} For every $i \in \left\{ 0, \dots, n-1 \right\} $, given an $(i-1)$-face $F_{i-1}$ and an $(i+1)$-face $F_{i+1}$ with $F_{i-1} \leq F_{i+1}$, the set $\left\{ F \in \cP : F_{i-1} < F < F_{i+1} \right\} $ has cardinality $2$.
\end{enumerate}
The diamond condition implies that given $i \in \left\{ 0, \dots, n-1 \right\} $, for every flag $\Phi$ of $\cP$ there exists a unique flag $\Phi^{i}$ such that $\Phi$ and $\Phi^{i}$ differ exactly in the face of rank $i$. 
In this situation we say that $\Phi$ and $\Phi^{i}$ are \emph{adjacent} or \emph{$i$-adjacent}, if we need to emphasise on $i$.
We also require that
 \begin{enumerate}[resume]
 \item \label{item:sfc} $\cP$ is \emph{strongly flag connected}.
\end{enumerate}
Meaning that for every two flags $\Phi$ and $\Psi$ there exists a sequence $\Phi=\Phi_{0}, \Phi_{1}, \dots, \Phi_{k} = \Psi$ such that every two consecutive flags are adjacent and $\Phi\cap\Psi \subset \Phi_{i}$ for every $i \in \left\{ 0, \dots, k \right\} $.

If $i_{1}, i_{2}, \dots, i_{k} \in \left\{ 0, \dots, n-1 \right\} $, we define recursively $\Phi^{i_{1}, \dots, i_{k}} = (\Phi^{i_{1}, \dots, i_{k-1}})^{i_{k}}$.

If $F$ and $G$ are two faces of a polytope such that $F \leq G$, the \emph{section} $G/F$ is the restriction of the order of $\cP$ to the set $\left\{ H \in \cP : F \leq H \leq G \right\} $.
Note that if $\rk(F) = i$ and $\rk(G) = j$ then the section $G/F$ is an abstract polytope of rank $j-i-1$.
If $F_{0}$ is a vertex, then the \emph{vertex-figure} at $F_{0}$ is the section $F_{n}/F_{0}$. 
We sometimes identify each face $F$ with the section $F/F_{-1}$. 
In particular, every facet $F_{n-1}$ can be identified with the section $F_{n-1}/F_{-1}$ of rank $n-1$.

The \emph{dual} of an abstract polytope $(\cP, \leq)$, commonly denoted by $\cP^{\ast}$, is the partially ordered set $(\cP, \leq_{\ast})$ where $F \leq_{\ast} G$ if and only if $F \geq G$. 

Given an abstract $n$-polytope $\cK$, an $(n+1)$-polytope $\cP$ is an \emph{extension} of $\cK$ if all the facets of $\cP$ are isomorphic to $\cK$.

For $i \in \left\{ 1, \dots, n-1 \right\} $, if $F$ is an $(i-1)$-face and $G$ is an $(i+2)$-face with $F \leq G$ then the section $G/F$ is a $2$-polytope. 
Therefore $G/F$ is isomorphic to a $p_{i}$-gon for some $p_{i} \in \left\{ 2, \dots, \infty \right\} $. 
If the number $p_{i}$ does not depend on the particular choice of $F$ and $G$ but only on $i$, then we say that  $\left\{ p_{1}, \dots, p_{n-1} \right\}$ is the \emph{Schläfli symbol} of $\cP$. 
In this situation sometimes we just say that $\cP$ is of \emph{type} $\left\{ p_{1}, \dots, p_{n-1} \right\} $.
Note that if $\cP$ is an $n$-polytope of type $\left\{ p_{1}, \dots, p_{n-1} \right\} $, then the the facets of $\cP$ are of type $\left\{ p_{1}, \dots, p_{n-2} \right\} $.
In particular, if $\cP$ is an extension of $\cK$, and $\cP$ has a well-defined Schläfli symbol, then all but the last entry of this symbol are determined by $\cK$.

An \emph{automorphism} of an abstract polytope $\cP$ is an order-preserving bijection $\gamma: \cP \to \cP$.
The group of automorphism of $\cP$ is denoted by $\aut(\cP)$.
The group $\aut(\cP)$ acts naturally on $\fl(\cP)$, the set of flags of $\cP$. 
This action satisfies \[\Psi^{i} \gamma = \left( \Psi \gamma \right)^{i}\] for every flag $\Psi$, $i \in \left\{ 0, \dots, n-1 \right\} $ and $\gamma \in \aut(\cP)$.
Moreover, as a consequence of the strong-flag-connectivity, the action  of $\aut(\cP)$ on $\fl(\cP)$ is free.

Let $\baseFlag = \left\{ F_{-1}, \dots, F_{n} \right\} $ be a \emph{base flag} of $\cP$  such that $\rk(F_{i}) = i$.
Let $\Gamma \leq \aut(\cP)$ and for $I \subset \left\{ 0, \dots, n-1 \right\} $ let $\Gamma_{I}$ denote the set-wise stabiliser of the chain $\left\{ F_{i} : i \not\in I \right\} \subset \baseFlag $. 
Note that for every pair is subsets $I, J \subset \left\{ 0, \dots, n-1 \right\} $ we have 
\begin{equation}\label{eq:intProperty}
	\Gamma_{I} \cap \Gamma_{J} = \Gamma_{I \cap J}
\end{equation}
We call this condition the \emph{intersection property} for $\Gamma$.

An abstract polytope is $\emph{regular}$ if the action of $\aut(\cP)$ on $\fl(\cP)$ is transitive (hence, regular).
Abstract regular polytopes have been traditionally the most studied family of polytopes. 
Most of their wide theory can be found in \cite{McMullenSchulte_2002_AbstractRegularPolytopes}.

A \emph{rooted polytope} is a pair $(\cP, \baseFlag)$ where $\cP$ is a polytope and $\baseFlag$ is a fixed base flag.
If $\cP$ is regular, then every two flags are equivalent under $\aut(\cP)$ and the choice of a particular base flag plays no relevant role. 
However, if $\cP$ is not regular, then the choice of the base flag is important.
See \cite{CunninghamPellicer_2018_OpenProblems$k$} for a discussion on rooted $k$-orbit polytopes.
 
If $\Proot$ is a regular rooted polytope, then for every $i \in \left\{ 0, \dots, n-1 \right\} $ there exists an automorphism $\rho_{i}$ such that \[\baseFlag\rho_{i} = \baseFlag^{i}.\]
We call the automorphisms $\rho_{0}, \dots, \rho_{n-1}$ the \emph{abstract reflections} (with respect to the base flag $\baseFlag$). 
It is easy to see that if $\cP$ is a regular $n$-polytope, then $\aut(\cP) = \left\langle \rho_{0}, \dots, \rho_{n-1} \right\rangle $.
It is important to remark that the group elements depend on $\baseFlag$. 
However, since $\aut(\cP)$ is transitive on flags, the choice of the base flag induces a group-automorphism of $\aut(\cP)$.
More precisely, let $\Phi$ and $\Psi$ be flags of a regular $n$-polytope $\cP$ and let $\rho_{0}, \dots, \rho_{n-1}$ and $\rho'_{0}, \dots, \rho'_{n-1}$ denote the abstract reflections with respect to $\Phi$ and $\Psi$ respectively. 
If $\gamma \in \aut(\cP)$ is such that $\Phi \gamma  = \Psi$, then $\rho'_{i} = \gamma^{-1} \rho_{i} \gamma$.

Note that every regular polytope has a well-defined Schläfli symbol.
For $\cP$  is a regular $n$-polytope of type $\left\{ p_{1}, \dots, p_{n-1} \right\} $, then the abstract reflections satisfy 
\begin{equation}\label{eq:relsRhos}
	\begin{aligned} 
		\rho_{i}^{2} &= \id && \text{for } i \in \left\{ 0, \dots, n-1 \right\} ,\\
		\left( \rho_{i} \rho_{j} \right)^{2} &= \id && \text{if } |i-j| \geq 2, \\
		\left( \rho_{i-1} \rho_{i}\right)^{p_{i}} &= \id && \text{for } i \in \left\{ 1, \dots, n-1 \right\}. 
	\end{aligned}
\end{equation}

If $\baseFlag = \left\{ F_{-1}, \dots, F_{n} \right\} $, the stabiliser of the chain $\left\{ F_{i} : i \not \in I \right\} $ is the group $\left\langle \rho_{i} : i \in I \right\rangle $.
It follows that for regular polytopes the intersection property in \cref{eq:intProperty} for $\aut(\cP)$ is equivalent to
\begin{equation}\label{eq:interPropReg}
 \left\langle \rho_{i} : i \in I \right\rangle \cap \left\langle \rho_{j}: j \in J \right\rangle = \left\langle \rho_{k} : k \in I \cap J \right\rangle   
\end{equation}
for every pair of sets $I, J \subset \left\{ 0, \dots, n-1 \right\} $.

A group $\left\langle \rho_{0}, \dots, \rho_{n-1} \right\rangle $ satisfying \cref{eq:relsRhos,eq:interPropReg} is a \emph{string C-group}. 
Clearly, the automorphism group of an abstract polytope is a string C-group.
One of the most remarkable facts in the theory of highly symmetric polytopes is the correspondence between string C-groups and abstract regular polytopes. 
To be precise, for every string C-group $\Gamma$, there exists an abstract regular polytope $\cP = \cP(\Gamma)$ such that $\aut(\cP) = \Gamma$. 
This fact was first proved by Schulte in \cite{Schulte_1980_RegulareInzidenzkomplexe_phdThesis} for \emph{regular incidence complexes}, which are structures slightly more general than abstract polytopes. 
A detailed proof can be found in \cite[Section 2E]{McMullenSchulte_2002_AbstractRegularPolytopes}. See also \cite{Schulte_1983_RegulareInzidenzkomplexe.Ii,Schulte_2018_RegularIncidenceComplexes}.

The correspondence mentioned above has been used to build families of abstract regular polytopes with prescribed desired properties. 
For instance in \cite{Schulte_1983_ArrangingRegularIncidence} and \cite{Schulte_1985_ExtensionsRegularComplexes} some universal constructions are explored. 
In another direction, in \cite{CameronFernandesLeemansMixer_2017_HighestRankPolytope, FernandesLeemans_2018_CGroupsHigh, LeemansMoerenhoutOReillyRegueiro_2017_ProjectiveLinearGroups, Pellicer_2008_CprGraphsRegular} some abstract regular polytopes with prescribed (interesting) groups are investigated. 
Of particular interest for this paper is the work in \cite{Danzer_1984_RegularIncidenceComplexes, Pellicer_2009_ExtensionsRegularPolytopes, Pellicer_2010_ExtensionsDuallyBipartite}, where the problem of determining the possible values of the last entry of the Schläfli symbol of a regular extension of a given regular polytope is addressed. 

\subsection{Rotary polytopes} 
Abstract regular polytopes are those with maximal degree of reflectional symmetry. 
A slightly weaker symmetry condition than regularity for an abstract polytope is to admit all possible rotational symmetries.
In a similar way as it has been done for maps (see \cite{Wilson_1978_RiemannSurfacesOver}, for example), we call these polytopes \emph{rotary polytopes}. 
In this section, we review some of the theory of rotary polytopes.
Most of this theory is developed in \cite{SchulteWeiss_1991_ChiralPolytopes}.

The \emph{flag-graph} of a polytope $\cP$, denoted by $\cG_{\cP}$, is the edge-coloured graph whose vertex-set is the set $\fl(\cP)$ of flags and two flags are connected by an edge of colour $i$ if they are $i$-adjacent. 
An abstract polytope $\cP$ is \emph{orientable} if $\cG_{\cP}$ is bipartite. 
Otherwise, $\cP$ is \emph{non-orientable}.
If $\Proot$ is a rooted orientable polytope, the set of flags on the same part as $\baseFlag$ are the \emph{white flags}.
We denote this set by $\fw(\cP)$.
The flags on the other part are the \emph{black flags}.
This is just another way of calling what in \cite{SchulteWeiss_1991_ChiralPolytopes} are called \emph{even} and \emph{odd} flags.
By convenience, if $\cP$ is non-orientable, we say that $\fw(\cP) = \fl (\cP)$.
In other words, every flag is white.

If $\cP$ is an abstract polytope, then the \emph{rotational group} $\autp(\cP) $ of $\cP$ is the subgroup of $\aut(\cP)$ that permutes the set of white flags. 
A polytope is \emph{rotary} if $\autp(\cP)$ acts transitively on the set of white flags. 
It is clear the a rotary non-orientable polytope is a regular polytope. 
Therefore, we restrict our discussion below to orientable polytopes.
Note that the choice of the base flag of $\cP$ plays a stronger role now. 
In particular, it defines the set of white flags.
In the discussion below it is assumed that $\cP$ is actually a rooted polytope $\Proot$.

If $\Proot$ is a  polytope, for every $i \in \left\{ 1, \dots, n-1 \right\} $ the flag $\baseFlag^{i,i-1}$ is a white flag.
Therefore, if $\cP$ is a rotary polytope, there exists an automorphism $\sigma_{i}$ such that \[\baseFlag \sigma_{i} = \baseFlag^{i, i-1}.\]
The automorphisms $\sigma_{1}, \dots, \sigma_{n-1}$ are called the \emph{abstract rotations} with respect to $\baseFlag$.
It is easy to see that $\autp(\cP) = \left\langle \sigma_{1}, \dots, \sigma_{n-1} \right\rangle $. 
This is also a good point to emphasise that the abstract rotations depend on the choice of the base flag.

If $\cP$ is a rotary $n$-polytope, then $\cP$ has a well-defined Schläfli symbol. 
If $\cP$ is of type $\left\{ p_{1}, \dots, p_{n-1} \right\} $ the automorphisms $\sigma_{1}, \dots, \sigma_{n-1}$ satisfy 
\begin{equation}\label{eq:relsSigmas}
 \begin{aligned}
		 	  \sigma_{i}^{p_{i}} &= \id, && \text{and} \\
	  (\sigma_{i} \sigma_{i+1} \cdots \sigma_{j})^{2} &= \id && \text{for}\ 1 \leq i < j \leq n-1.
 \end{aligned}
\end{equation}

Sometimes it is useful to consider an alternative set of generators for $\autp(\cP)$. 
For $i,j \in \{0, \dots, n-1\}$ with $i <j$ we define the automorphisms 
\[\tau_{i,j} = \sigma_{i+1}\cdots \sigma_{j}.\] 
Note that this is a small change with respect to the notation of \cite[Eq. 5]{SchulteWeiss_1991_ChiralPolytopes}. 
What they call $\tau_{i,j}$ for us is $\tau_{i-1,j}$. 
Observe that $\tau_{i-1,i} = \sigma_{i}$ for $i \in \{1, \dots, n-1\}$. It is also convenient to define $\tau_{j,i} = \tau_{i,j}^{-1}$ for $i < j$ and $\tau_{-1,j}= \tau_{i,n} = \tau_{i,i} = \id$ for every $i,j \in \{0, \dots, n-1\}$. 
In particular, we have that $\langle \tau_{i,j} : i,j \in \{0, \dots, n-1\} \rangle = \langle \sigma_{1}, \dots, \sigma_{n-1} \rangle$. 
We also have \[\baseFlag\tau_{i,j}= \baseFlag^{j,i}.\]

Moreover, if $\baseFlag = \left\{ F_{-1}, \dots, F_{n} \right\} $, and $I \subset \left\{ 0, \dots, n-1 \right\} $, then the stabiliser of the chain $\left\{ F_{i} : i \not\in I  \right\} $ is the group $\left\langle \tau_{i,j} : i,j \in I \right\rangle $.
It follows that the intersection property in \cref{eq:intProperty} for $\autp(\cP)$ can be written as
\begin{equation}\label{eq:intPropertyChiral}
 \left\langle \tau_{i,j} : i,j \in I \right\rangle \cap \left\langle \tau_{i,j} : i,j \in J \right\rangle = \left\langle \tau_{i,j} : i,j \in I \cap J \right\rangle.  
\end{equation}

If $\cP$ is a regular polytope with automorphism group $\aut(\cP) = \left\langle \rho_{0}, \dots, \rho_{n-1} \right\rangle $, then $\cP$ is rotary with $\autp(\cP) = \left\langle \sigma_{1}, \dots, \sigma_{n-1} \right\rangle $ where $\sigma_{i} = \rho_{i-1} \rho_{i}$. 
If $\cP$ is also orientable, then $\autp(\cP)$ is a proper subgroup of $\aut(\cP)$ of index $2$.
Furthermore, $\autp(\cP)$ induces two flag-orbits, namely, the white flags and the black flags.

If $\cP$ is rotary but not regular, then $\aut(\cP) = \autp(\cP)$ and this group induces precisely two orbits in flags in such a way that adjacent flags belong to different oribts. 
In this case we say that $\cP$ is \emph{chiral}. 
Chiral polytopes were introduced by Schulte and Weiss in \cite{SchulteWeiss_1991_ChiralPolytopes} as a combinatorial generalisation of Coxeter's twisted honeycombs in \cite{Coxeter_1970_TwistedHoneycombs}.

If $(\cP,\baseFlag)$ is a rooted chiral polytope, the \emph{enantiomorphic form of $\cP$}, usually denoted by $\bar{\cP}$, is the rooted polytope $(\cP, \baseFlag^{0})$.
In the classic development of the theory of chiral polytopes, the enantiomorphic form of $\cP$ is usually thought as the mirror image of $\cP$; as a polytope that is different (but isomorphic) from $\cP$.
However, when treated as rooted polytopes it is clear that the only difference is the choice of the base flag.
The underlying partially ordered set is exactly the same. 
For a traditional but detailed discussion about enantiomorphic forms of chiral polytopes we suggest \cite[Section 3]{SchulteWeiss_1994_ChiralityProjectiveLinear}.

The automorphism group of $\bar{\cP}$ is generated by the automorphisms $\sigma'_{1}, \dots, \sigma'_{n-1}$ where $(\baseFlag^{0}) \sigma'_{i} = (\baseFlag^{0})^{i,i-1}$.
It is easy to verify that $\sigma'_{1} = \sigma_{1}^{-1} $, $\sigma_{2} = \sigma_{1}^{2} \sigma_{2}$ and for $i \geq 3$, $\sigma'_{i} = \sigma_{i}$. 
If $\cP$ is orientably regular then conjugation by $\rho_{0}$ defines a group automorphism $\rho: \autp(\cP) \to \autp(\cP)$ that maps $\sigma_{i}$ to $\sigma'_{i}$. 

The relations in \cref{eq:relsSigmas} together with the intersection property in \cref{eq:intPropertyChiral} characterise the rotation groups of the rotary polytopes.
Moreover, the existence of the group-automorphism $\rho: \autp(\cP) \to \autp(\cP)$ mentioned above determines whether or not the rotary polytope is regular.
More precisely, the following result hold.

\begin{thm}[{Theorem 1 in \cite{SchulteWeiss_1991_ChiralPolytopes}}]\label{thm:chiralGroups}
Let $3\leq n$, $2 < p_{1}, \dots, p_{n-1} \leq \infty$ and $\Gamma=\langle \sigma_{1}, \dots, \sigma_{n-1} \rangle$. For every $i,j \in \{-1, \dots, n\}$, with $i \neq j$ define
\[\tau_{i,j}=
\begin{cases}
	\id & \text{if}\ i<j\ \text{and}\ i=-1\ \text{or}\ j=n,\\ 
  \sigma_{i+1} \cdots \sigma_{j} & \text{if}\ 0\leq i<j\leq n-1, \\
  \sigma^{-1}_{j} \cdots \sigma^{-1}_{i+1} & \text{if}\ 0\leq j<i\leq n-1. 
\end{cases}
\]

Assume that $\Gamma$ satisfies the relations in  \cref{eq:relsSigmas}. Assume also that \cref{eq:intPropertyChiral} holds. Then
\begin{enumerate}
  \item\label[part]{part:existence} There exists a rotary polytope $\cP=\cP(\Gamma)$ such that $\autp(\cP)=\Gamma$ and $\sigma_{1}, \dots, \sigma_{n-1}$ act as abstract rotations for some flag of $\cP$.
  \item\label[part]{part:facets} $\cP$ is of type $\{p_{1}, \dots, p_{n-1}\}$. The facets and vertex-figures of $\cP$ are isomorphic to $\cP(\langle\sigma_{1}, \dots, \sigma_{n-2}\rangle)$ and $\cP(\langle \sigma_{2}, \dots, \sigma_{n-1})\rangle$, respectively. In general if $n \geq 4$ and $1 \leq k < l \leq n-1$, $F$ is a $(k-2)$-face and $G$ is an incident $(l+1)$-face, then the section $G/F$ is a rotary $(l-k+2)$-polytope isomorphic to $\cP(\langle \sigma_{k} \dots, \sigma_{l} \rangle)$. 
  \item\label[part]{part:chirality} $\cP$ is orientably regular if and only if there exists an involutory group automorphism $\rho:\Gamma \to \Gamma$ such that $\rho: \sigma_{1} \mapsto \sigma^{-1}_{1}$, $\rho: \sigma_{2} \mapsto \sigma_{1}^{2}\sigma_{2}$ and $\rho: \sigma_{i}\mapsto \sigma_{i}$ for $i\geq 3$.
\end{enumerate}
\end{thm}

In \cref{sec:dually} we use $\{\sigma_{1}, \dots, \sigma_{n-2}, \tau \}$, with $\tau = \sigma_{n-2}\sigma_{n-1}$, as an alternative generating set for $\autp(\cP)$. 
The following lemma describe the relations in \cref{eq:relsSigmas} in terms of these generators.

\begin{lem}[{Lemma 2.1 in \cite{Pellicer_2010_ConstructionHigherRank}}]\label{lem:relsOfExtensionsSigmasAndTau}
  Let $n \geq 4$ and $\Gamma = \langle \sigma_{1}, \dots, \sigma_{n-1} \rangle$ be a group with the property that the subgroup $\Gamma_{n-1} = \langle \sigma_{1}, \dots, \sigma_{n-2} \rangle$ satisfies the relations in  \cref{eq:relsSigmas}. 
  Let $\tau = \sigma_{n-2} \sigma_{n-1}$. 
  Then the set of relations in \cref{eq:relsSigmas} is equivalent to the set of relations
  \begin{equation}\label{eq:relsOfExtensionsSigmasAndTau}
    \begin{aligned}
      \tau^{2} &= \id, \\
      \tau \sigma_{n-3}\tau &= \sigma_{n-3}^{-1},\\
      \tau \sigma_{n-4} \tau &= \sigma_{n-4}\sigma_{n-2}^{2},\\
      \tau \sigma_{i}\tau &= \sigma_{i} && \text{for}\ 1 \leq i < n-4.
    \end{aligned}
  \end{equation}
\end{lem}

% 
% % % % % % % % % % % % % % % % % % % % % % % % % % % % % % % % % % % % % % % % % % % % % % % 
% % Connection group and aut
% % % % % % % % % % % % % % % % % % % % % % % % % % % % % % % % % % % % % % % % % % % % % % % 
\subsection{Maniplexes, connection group, coverings and quotients}.
Sometimes it is useful to work in a wider universe than that of abstract polytopes. 
The universe of maniplexes will work for our purposes.
Maniplexes were introduced by Wilson in \cite{Wilson_2012_ManiplexesPart1} as a combinatorial generalisation of maps to higher dimensions.  
A maniplex of rank $n$ (or $n$-maniplex) is a connected graph resembling the flag graph of a polytope.
Formally, a \emph{maniplex} or rank $n$ (or $n$-maniplex) $\cM$ is a connected, $n$-valent graph whose nodes are the \emph{flags} of $\cM$ and whose edges are properly coloured with $\left\{ 0, \dots n-1 \right\} $, meaning that for every $i \in \left\{ 0, \dots, n-1 \right\} $, the edges with label $i$ determine a perfect matching of $\cM$. 
The set of flags of $\cM$ is denoted by $\fl(\cM)$.
We also require that if $i,j \in \left\{ 0, \dots, n-1 \right\} $ and $|i-j|\geq 2$, then the connected components of $\cM$ determined by the edges of colours $i$ and $j$ are alternating squares.
Note the slight change of notation: what we call an $n$-maniplex is called an $(n-1)$-maniplex in \cite{ CunninghamPellicer_2018_OpenProblems$k$, Wilson_2012_ManiplexesPart1}. 

Every $n$-polytope $\cP$ determines an $n$-maniplex through its flag-graph and it is usually safe $\cP$ with the maniplex induced by $\cG_{\cP}$
In general, not every maniplex induces a polytope.
For a discussion about polytopality of maniplexes and examples of maniplexes that are not polytopes see \cite{GarzaVargasHubard_2018_PolytopalityManiplexes}.

A facet of an $n$-maniplex $\cM$ is a connected component of the graph resulting of removing the edges of colour $n-1$ from $\cM$. 

An \emph{automorphism} of a maniplex $\cM$ is a colour-preserving graph automorphism. 
The {automorphism group} of $\cM$ is denoted by $\aut(\cM)$. 
Note that the group $\aut(\cM)$ acts freely on $\fl(\cM)$. 
The notions of rooted, regular, rotary and chiral maniplexes extend in a naturally from polytopes to maniplexes.

Note that if $\cM$ is a maniplex and $i \in \{0, \dots, n-1\}$, then the edges of colour $i$ define an involutory permutation $r_{i}$ such that $r_{i}(\Phi) = \Phi^{i}$ for every $\Phi \in \fl(\cM)$. 
The \emph{connection group} of $\cM$ is group $\con(\cM) = \left\langle r_{0}, \dots, r_{n-1} \right\rangle$ (also known  as the \emph{monodromy group} in \cite{MonsonPellicerWilliams_2014_MixingMonodromyAbstract} and elsewhere). 
Note that in general the connection elements do not define automorphism.
Also observe that in this paper we use a left action for connection elements. 

If $\cM$ is a maniplex, then the permutations $r_{0}, \dots, r_{n-1}$ have the following properties: 
\begin{enumerate}
 \item\label{item:mani_1} The group $ \con(\cM) = \left\langle r_{0}, \dots, r_{n-1} \right\rangle $ acts transitively on the set $\fl(\cM)$. 
 \item\label{item:mani_2} Each permutation $r_{i}$ is fixed-point-free.
 \item\label{item:mani_3} If $i \neq j$ and $\Phi$ is any flag, then $r_{i} \Phi \neq r_{j} \Phi$.
 \item\label{item:mani_4} If $|i-j| \geq 2$, then  $r_{i} $ and $ r_{j} $ commute.
\end{enumerate}

Conversely, if $\fl$ is a set and $r_{0}, \dots r_{n-1}$ are permutations on $\fl$ satisfying \cref{item:mani_1,item:mani_2,item:mani_3,item:mani_4} above, then they define a maniplex $\cM$ with $\fl(\cM) = \fl$ and $\con(\cM) = \left\langle r_{0}, \dots, r_{n-1} \right\rangle $.
In  other words, we can define a maniplex by its underlying graph or through its connection group (see \cite{PellicerPotocnikToledo_2019_ExistenceResultTwo}, for example).
The facets of a maniplex are precisely the orbits of a flag under the action of $\left\langle r_{0}, \dots, r_{n-1} \right\rangle $.

A permutation $\gamma : \fl(\cM) \to \fl(\cM)$ induces an automorphism of $\cM$ if and only if 
\[\left( r_{i} \Phi \right) \gamma  = r_{i} (\Phi) \gamma\] for every flag $\Phi$.
 
It is well known that if $\cM$ is a regular maniplex then $\con(\cM) \cong \aut(\cM)$ with the isomorphism mapping $r_{i}$ to $\rho_{i}$ (see \cite[Theorem 3.9]{MonsonPellicerWilliams_2014_MixingMonodromyAbstract}). 
If $\cM$ is a rotary maniplex, then the action of $\autp(\cP)$ and the action of $\con(\cP)$ have an interesting relation.
This relation is described in the following proposition (see \cite[Lemma 2.5 and Proposition 2.7]{Pellicer_2010_ConstructionHigherRank}).

\begin{prop}\label{prop:autpAsMonPChir}
  Let $\cP$ be a rotary $n$-polytope with base flag $\baseFlag$. Let $\autp(\cP)=\langle \sigma_{1}, \dots,  \sigma_{n-1} \rangle$ and $\con(\cP)=\langle r_{0}, \dots, r_{n-1} \rangle$ be the rotation and connection groups of $\cP$ respectively. For $1 \leq i \leq n-1$, define $s_{i}=r_{i-1}r_{i}$ and consider the \emph{even connection group} $\conp(\cP) = \langle s_{1}, \dots, s_{n-1} \rangle$ of $\cP$. Then the following hold
\begin{enumerate}
  \item\label[part]{part:jumpsChir} For every $i_{1}, \dots,i_{k} \in \{1, \dots, n-1\}$ \[s_{i_{1}}\cdots s_{i_{k}}\baseFlag=\baseFlag\sigma_{i_{1}}\cdots\sigma_{i_{k}}.\]
  \item\label[part]{part:relationsChir} An element $s_{i_{1}}\cdots s_{i_{k}}\in\conp(\cP)$ fixes the base flag if and only if $\sigma_{i_{1}}\cdots \sigma_{i_{k}} = \epsilon$. 
  In this situation, $s_{i_{1}} \cdots s_{i_{k}}$ stabilises every white flag. 
  \item\label[part]{part:isomorphismChir} If $\fw(\cP)$ denotes the set of white flags of $\cP$, then there is an isomorphism $f: \conw(\cP) \to \autp(\cP)$ with $\conw(\cP)= \langle \bar{s}_{1}, \dots, \bar{s}_{n-1} \rangle$, where $\bar{s}_{i}$ denotes the permutation of $\fw(\cP)$ induced by $s_{i}$. 
  This isomorphism maps $\bar{s}_{i}$ to $\sigma_{i}$ for every $i \in \{1, \dots, n-1\}$.
\end{enumerate} 
\end{prop}

Note that unlike the regular case, where there is a isomorphism from $\con(\cP)$ to $\aut(\cP)$ mapping $r_{i}$ to $\rho_{i}$, if $\cP$ is chiral, the mapping $g: \conp(\cP) \to \autp(\cP)$ defined by $s_{i} \mapsto \sigma_{i}$ is not an isomorphism. 
By \cref{part:isomorphismChir} of \cref{prop:autpAsMonPChir} this mapping is well defined, but in general it is not injective. 
In other words, there are non-trivial elements of $\conp(\cP)$ that fix every white flag of $\cP$. 
The subgroup of $\conp(\cP)$ containing all such elements (i.e. the kernel of $g$) is called the \emph{chirality group}. 
For some uses and properties of the chirality group see \cite[Section 3]{Cunningham_2012_MixingChiralPolytopes} and \cite[Section 7]{MonsonPellicerWilliams_2014_MixingMonodromyAbstract}.

To avoid confusion, we will try to avoid the use of $\conp(\cP)$ and instead use the group $\conw(\cP)$, which is the permutation group on $\fw(\cP)$ induced by the action of $\conp(\cP)$.
Since we will only use the action of $\conp(\cP)$ on white flags,  it is safe to abuse notation an identify $s_{i} = r_{i-1} r_{i} \in \conp(\cP)$ with $\bar{s_{i}} \in \conw(\cP)$, the permutation induced by $s_{i}$ on the set $\fw(\cP)$.

If $\cP$ and $\cQ$ are $n$-polytopes, we say that $\cP$ \emph{covers} $\cQ$ if there exists a surjective function $\phi: \cP \to \cQ$ such that $\phi$ preserves incidence, rank and flag-adjacency. 
In this situation, we say that $\phi$ is a \emph{covering} from $\cP$ to $\cQ$.
If $\Proot$ and $\Proot[\cQ][\Psi_{0}]$ are rooted polytopes, then $\Proot$ covers $\Proot[\cQ][\Psi_{0}]$ if there is a covering from $\cP$ to $\cQ$ that maps $\baseFlag$ to $\Psi_{0}$.
If $\cP$ covers $\cQ$ we also say that $\cQ$ is a \emph{quotient} of $\cP$.
If $\left( \cP, \baseFlag \right)$ covers $\left( \cQ, \Psi_{0} \right)$, then there exists a colour-preserving graph homomorphism from $\cG_{\cP} $ to $\cG_{\cQ}$ mapping $\baseFlag$ to $\Psi_{0}$. 
Conversely, if such a graph homomorphism exists, then $\left( \cP, \baseFlag \right)$ covers $\left( \cQ, \Psi_{0} \right)$ (see \cite[Proposition 3.1]{CunninghamPellicer_2018_OpenProblems$k$}). 
The previous observation allows us to extend naturally the notion of covering to maniplexes and rooted maniplexes. Moreover, for rotary maniplexes we have the following result. This is a consequence of \cref{part:isomorphismChir} of \cref{prop:autpAsMonPChir}.

\begin{prop}\label{prop:covers}
 Let $(\cM, \Phi_{0})$ and $(\cN, \Psi_{0})$ be two rooted rotary $n$- maniplexes. Let $\autp(\cM) = \left\langle \sigma_{1}, \dots, \sigma_{n-1} \right\rangle $, $\autp(\cN) = \left\langle \sigma'_{1}, \dots, \sigma_{n-1} \right\rangle $, $\conp(\cM) = \left\langle s_{1}, \dots, s_{n-1} \right\rangle $, $\conp(\cN) = \left\langle s'_{1}, \dots, s'_{n-1} \right\rangle $  denote their rotation and even connection groups.
 Then the following are equivalent. 
 \begin{enumerate}
  \item \label[part]{part:covers_graph} $(\cM, \Phi_{0})$ covers $(\cN,\Psi_{0})$.
% %   \item \label{item:covers_graph} There exists a colour-preserving graph automorphism $h: \cM \to \cN$ mapping $\Phi_{0}$ to $\Psi_{0}$.
  \item \label[part]{part:covers_con} There exists a group-homomorphism $\eta: \conp(\cM) \to \conp(\cN)$ mapping $s_{i}$ to $s'_{i}$.
  \item \label[part]{part:covers_aut} There exists a group-homomorphism $\bar{\eta}: \autp(\cM) \to \autp(\cN)$ mapping $\sigma_{i}$ to $\sigma'_{i}$.
 \end{enumerate}
\end{prop}
% \begin{proof}
%  Clearly \cref{part:covers_graph} implies \cref{part:covers_con}. 
%  The graph homomorphism from $\cM$ to $\cN$ induces a group-homomorphism from $\eta: \con(\cM)\to \con(\cN)$. 
%  By restricting $\eta$ to $\conp(\cM)$ we obtain the desired homomorphism.
%  
%  By \cref{part:relationsChir} of \cref{prop:autpAsMonPChir}, to see that \cref{part:covers_con} implies \cref{part:covers_aut} it is enough to prove that is $s_{i_{1}}\cdots s_{i_{k}} \in \conp(\cM)$ stabilises $\Phi_{0}$, then $s'_{i_{1}}\cdots s'_{i_{k}}$ stabilises $\Psi_{0}$. 
%  
% \end{proof}

Let $\cM$ and $\cN$ be rotary $n$-maniplexes with rotation groups $\autp(\cM) = \left\langle \sigma_{1}, \dots, \sigma_{n-1} \right\rangle $ and $\autp(\cN) = \left\langle \sigma'_{1}, \dots, \sigma'_{n-1} \right\rangle $.
We define the group
\begin{equation}\label{eq:mix}
\autp(\cM) \mix \autp(\cN) = \left\langle (\sigma_{1}, \sigma'_{1}), \dots (\sigma_{n-1}, \sigma'_{n-1})  \right\rangle \leq \autp(\cM)\times \autp(\cN).
\end{equation}

A straightforward consequence of \cref{part:covers_aut} is the following.
\begin{lem}\label{lem:covers}
 If $(\cM, \Phi_{0})$ and $(\cN, \Psi_{0})$ are rooted maniplexes such that $(\cM, \Phi_{0})$ covers $(\cN, \Psi_{0})$, then $\autp(\cM) \mix \autp(\cN) \cong \autp(\cM)$.
\end{lem}

Note that if $\cM$ is of type $\left\{ q_{1}, \dots, q_{n-1} \right\} $ and $\cN$ is of type $\left\{ q'_{1}, \dots, q'_{n-1} \right\} $, then the group $\autp(\cM) \mix \autp(\cN)$ satisfy the relations in \cref{eq:relsSigmas} for $p_{i}= \lcm(q_{i}, q'_{i})$.
However, in general the group $\autp(\cM) \mix \autp(\cN)$ will not satisfy the intersection property in \cref{eq:intPropertyChiral}, even when $\cM$ and $\cN$ are polytopes. 
The following results can be used to prove that the group $\autp(\cM) \mix \autp(\cN)$ has the intersection propety under certain conditions.

\begin{lem}
[{Lemma 3.3 in \cite{DAzevedoJonesSchulte_2011_ConstructionsChiralPolytopes}}]
\label{lem:preQuotientCriterion}
Let $\Gamma = \left\langle \sigma_{1}, \dots, \sigma_{n-1} \right\rangle $ be a group satisfying \cref{eq:relsSigmas}, and let $\Lambda = \left\langle \lambda_{1}, \dots, \lambda_{n-1} \right\rangle $ be a group satisfying \cref{eq:relsRhos} and the intersectio property in \cref{eq:intPropertyChiral}. 
If the mapping $\sigma_{i} \mapsto \lambda_{i}$ for $i \in \left\{ 1, \dots, n-1 \right\} $ induces a homomorphism $\pi: \Gamma \to \Lambda$, which is one-to-one on $\left\langle \sigma_{1}, \dots, \sigma_{n-1} \right\rangle $, then $\Gamma$ also has the intersection property.
\end{lem}

If $\cP$ is rotary polytope and $\cM$ is a rotary maniplex such the facets of $\cP$ cover the facets of $\cM$, then the mapping $\pi: \autp(\cP) \mix \autp(\cM) \to \autp(\cP)$ defined by $(\sigma_{i}, \sigma'_{i}) \mapsto \sigma_{i}$ satisfy the hypotesis of \cref{lem:preQuotientCriterion}. 
In other words, we have the following result, which is essentially \cite[Lemma 3.3]{DAzevedoJonesSchulte_2011_ConstructionsChiralPolytopes}

\begin{lem}\label{lem:quotientCriterion}
 If $\cP$ is a rotary polytope and $\cM$ is a rotary maniplex so that the facets of $\cP$ cover the facets of $\cM$, then the group $\autp(\cP) \mix \autp(\cM)$ has the intersection property in \cref{eq:intPropertyChiral}.
\end{lem}

% GPR-graphs
% % I-components
% % Cayley+
% %Theorem of GPR-graphs
\subsection{GPR-graphs}
A \emph{general permutation representation graph} (or just \emph{GPR-graph}) is a directed graph encoding all the information of the rotation group of a rotary polytope.
They were introduced in \cite{PellicerWeiss_2010_GeneralizedCprGraphs} as a generalisation of CPR-graphs (C-group premutation representation graphs).
In this section we review the definition and basic properties of such graphs.

Let $\cP$ be a rotary $n$-polytope with rotation group $\autp(\cP)=\left\langle \sigma_{1}, \dots, \sigma_{n-1} \right\rangle $. 
Let $m \in \bN$ and let $\phi: \autp(\cK) \to S_{m}$ be an embedding of the group $\autp(\cK)$ into a symmetric group $S_{m}$. 
The \emph{GPR-graph} associated to $\phi$ is the directed labelled multigraph (parallel edges are allowed) whose vertices are $\{1, \dots, m\}$ and for which there is an arrow from $s$ to $t$, with label $k$, whenever $\phi(\sigma_{k})$ maps $s$ to $t$.

We  call \emph{$k$-arrows} the arrows labelled with $k$. 
Usually the embedding is given by a known action of $\aut(\cK)$ and it can be omitted. We also omit loops, so a point of $\{1,\dots, m\}$ is understood to be fixed by $\phi(\sigma_{k})$ if and only if it has no $k$-arrows starting on it. 
A connected component of arrows with labels in $I\subset\{1, \dots, n-1\}$ is called an \emph{$I$-component}, and if $I=\{k\}$ for some $k$, then it is called a \emph{$k$-component}.
Observe that an $I$-component consists of one orbit of points in $\left\{ 1, \dots, m \right\}$ under $\langle \phi(\sigma_{i}): i \in I \rangle$.

Note that a GPR-graph of a rotary polytope $\cP$ determines the automorphism group of $\cP$ (and hence, it determines $\cP$). 
Observe also that the group $\aut(\cP)$ acts on the vertices of any GPR-graph of $\cP$ via the embedding $\phi$. 
Moreover, every path from a vertex $u$ to a vertex $v$ of a GPR-graph determines a word on $\{\sigma_{1}, \dots \sigma_{n-1}\} \cup \{\sigma^{-1}_{1}, \dots, \sigma^{-1}_{n-1}\}$ and hence, an element $\alpha$ of $\autp(\cP)$. 
The element $\alpha$ satisfies that $\phi(\alpha)$ maps $u$ to $v$. 
In general different paths determine different elements of $\autp(\cP)$.

If $(\cP,\Phi_{0})$ is a rooted rotary polytope, we may consider the isomophism described in \cref{part:isomorphismChir} of \cref{prop:autpAsMonPChir} as an embedding of $\autp(\cP)$ into the permutation group $\conw(\cP)$ on the set of white flags given by the (left) action of $\conp(\cP)$. 
The GPR-graph associated to this embedding is called the \emph{Cayley GPR-graph} of $\cP$ and it is denoted by $\cay(\cP)$.

Since the action of $\conp(\cP)$ on the set of white flags is free (see \cref{part:relationsChir} of \cref{prop:autpAsMonPChir}), the following proposition is immediate.

\begin{prop}(Proposition 5 in \cite{CunninghamPellicer_2014_ChiralExtensionsChiral})\label{prop:faithfulCay}
  Let $\cay(\cP)$ be the Cayley GPR-graph of a rotary polytope $\cP$ and let $s$ and $t$ be vertices of $\cay(\cP)$. There exists a unique element $w \in \conp(\cP)$ mapping $s$ to $t$. In particular, every element of $\conp(\cP)$ determined by a path from $s$ to $t$ is equal to $w$.
\end{prop}

Given an embedding $\phi$ of $\autp(\cP)$ into $S_{m}$, there exists a natural embedding $\phi^{d}$ of $\autp(\cK)$ into the direct product  $S_{m} \times \dots \times S_{m}$ of $d$ copies of $S_{m}$. This embedding is given by $\phi^{d}:\sigma_{i}\mapsto(\phi(\sigma_{i}), \dots, \phi(\sigma_{i}))$.
By considering this embedding, the following proposition is obvious.

\begin{prop}\label{prop:CopiesOfGPR}
	Let $G_{1}, \dots, G_{d}$ be isomorphic copies of a GPR-graph of a rotary polytope $\cP$. Then the disjoint union of $G_{1}, \dots, G_{d}$ is a GPR-graph of $\cP$.
\end{prop}

We finish this review of GPR-graphs with the following result, which will be  useful in later sections. 

\begin{thm}[{Theorem 8 in \cite{CunninghamPellicer_2014_ChiralExtensionsChiral}}]\label{thm:GPRgraphs}
	Let $G$ be a directed graph with arrows labelled $1, \dots, n$. Let $G_{1}, \dots G_{d}$ be the $\{1,2, \dots, n-1\}$-components of $G$. Assume also that 
	\begin{enumerate}
		\item\label[part]{part:GPR_facets} $G_{1}, \dots G_{d}$ are isomorphic (as labeled directed graphs) to the Cayley GPR-graph of a fixed chiral $n$-polytope $\cK$ with regular facets.
		\item \label[part]{part:GPR_rels} For $k \in \{1, \dots, n-1\}$, the action of $(\sigma_{k}\cdots\sigma_{n})^{2}$ on the vertex set of $G$ is trivial, where $\sigma_{i}$ is the permutation determined by all arrows of label $i$.
		\item \label[part]{part:GPR_intersection1} $\langle \sigma_{1}, \dots, \sigma_{n-1} \rangle \cap \langle \sigma_{n} \rangle = \{\epsilon\}$.
		\item \label[part]{part:GPR_intersection2} For every $k \in \{2, \dots, n-1\}$ there exists a $\{1, \dots, n-1\}$-component $G_{i_{k}}$ and a $\{k, \dots, n\}$-component $D_{k}$ such that $G_{i_{k}} \cap D_{k}$ is a nonempty $\{k, \dots, n-1\}$-component.
	\end{enumerate}
	Then $G$ is a GPR-graph of a chiral $(n+1)$-polytope $\cP$ whose facets are isomorphic to $\cK$.
\end{thm}
\section{Extensions of dually bipartite chiral polytopes} \label{sec:dually}

In this section we give a construction of chiral extensions of dually bipartite chiral polytopes with regular facets. 
To be precise, given a finite dually-bipartite chiral polytope $\cK$ and $s \in \bN$, we construct a graph $\gpr$ such that $\gpr$ is a GPR-graph of a finite chiral extension  $\cP$ of $\cK$. 
To be precise, we prove the following theorem.
% Moreover, if $\cP$ has Schläfli symbol $\left\{ p_{1}, \dots, p_{n-1}, p_{n} \right\} $, then $s | p_{n}$. 
% In particular, since $s$ is arbitrary, we have the following result as a consequence.

\begin{thm}\label{thm:ManyExtensionsDBP}
 Let$n \geq 3$ and let $\cK$ be a finite dually-bipartite chiral $n$-polytope with regular facets.
 Let $s \in \bN$. 
 Then there exists a finite chiral extension $\cP$ of $\cK$ such that if $\cP$ is of type $\left\{ p_{1}, \dots, p_{n-1}, p_{n} \right\} $, then $2s | p_{n}$.
\end{thm}

In particular, since $s$ is arbitrary, we have the following result as a consequence.

\begin{coro}\label{coro:ManyExtensionsDBP}
 Let $n \geq 3$. Every finite dually-bipartite chiral $n$-polytope with regular facets admits infinitely many non-isomorphic chiral extensions.
\end{coro}

Recall that an $n$-polytope $\cK$ is dually bipartite if its facets admit a colouring with two colours in such a way that facets incident to a common $(n-2)$-face of $\cK$ have different colours. 

Let $\cK$ be a dually bipartite chiral $n$-polytope with regular facets. 
Let $\fw(\cK)$ denote the set of white flags of $\cK$ and for $i \in \{1, \dots, n-1\}$ let $s_{i}$ denote the permutation of $\fw(\cK)$ induced by the connection element $r_{i-1}r_{i}$.
Let $\conw(\cK)$ denote the permutation group of $\fw(\cK)$ generated by $\{s_{i} : 1 \leq i \leq n-1 \}$.
As discussed before, $\conw(\cK) \cong \autp(\cK)$.

Let $\cK_{n-1}$ denote the set of facets of $\cK$ and $c:\cK_{n-1} \to \{1,-1\} $ a colouring as the one described above. 
Observe that $c$ induces a colouring $\bar{c}: \fw(\cK) \to \{1,-1\}$ by assigning to each white flag the colour of its facet. 
Consider the following remark.

\begin{rem} \label{rem:inducedColouring}
  Let $\cK$ and $\bar{c}$ be as above. If $\Psi \in \langle s_{1}, \dots, s_{n-2} \rangle \Phi$, then $\bar{c}(\Psi) = \bar{c}(\Phi)$.
\end{rem}

We  use this property of the colouring $\bar{c}$ in our construction. Another important remark about dually bipartite polytopes is the following. 

\begin{rem}\label{rem:LastEntryDBP}
  If $\cK$ is a dually bipartite polytope of type $\{p_{1}, \dots, p_{n-1}\}$. If $p_{n-1} < \infty$, then $p_{n-1}$ is even. 
\end{rem}
 
% \cref{rem:LastEntryDBP} follows from the fact that the colouring of facets associated to $\cK$ induces a proper $2$-colouring (in the sense of graphs) on the polygonal co-face at each $(n-3)$-face. 

The idea of our construction is very similar to that of \cite{CunninghamPellicer_2014_ChiralExtensionsChiral}.
Let $\cK$ be a dually bipartite chiral $n$-polytope with regular facets. 
Take $s \in \bN$ and let $G_{1}, \dots G_{2s}$ be $2s$ copies of $\cay(\cK)$, constructed as follows.
 For each $\ell \in \bZ_{2s}$, the vertices of the graph $G_{\ell}$ will be pairs labelled by $(\Phi, \ell)$ where $\Phi \in \fw(\cK)$. 
 Note that for $k \in \{1, \dots, n-1\}$, there is an arrow labelled with $k$ (\emph{$k$-arrow}) from $(\Phi,\ell)$ to $(\Psi, \ell)$ if and only if $\Psi = s_{k}\Phi$.  
 For $I \subset \{1, \dots, n-1\}$, the $I$-component of a vertex $v$ is the connected component containing $v$ after removing the arrows whose labels do not belong to $I$.
 Observe that \cref{rem:inducedColouring} implies that the colouring $\bar{c}$ is such that if $v$ and $u$ belong to the same $\{1, \dots, n-2\}$-component, then $\bar{c}(u) = \bar{c}(v)$.
We assume that if $\baseFlag$ is base flag of $\cK$, then $ \bar{c}(\baseFlag)=1$.
The strategy is to define a matching $M$ on the vertices of the disjoint union of $G_{1}, \dots G_{2s}$. 
Then consider the involutory permutation $t$ that results from swapping the endpoints of every edge of $M$. 
We take $s_{n} = s^{-1}_{n-1} t $ and then use \cref{thm:GPRgraphs} to prove that the resulting graph is a GPR-graph of a chiral extension of $\cK$. 
Then we will explore the properties of the resulting extension. 
% We will use the group $\conp(\cK)$ instead of the group $\autp(\cK)$ (as in \cite{cunninghamPellicer_ChiralExtOfChiralPolytopes} and  \cref{subsec:chirExt_Finite}) for consistency with later constructions in this chapter. Recall that by \cref{prop:autpAsMonPChir} the two actions on the set of white flags are equivalent.
% 

We  define $M$ in several steps.
\begin{enumerate}%[label={\arabic*.}]
\item\label[step]{step:baseFlag} Add an edge between $(\baseFlag, \ell )$ and $(\baseFlag, \ell+(-1)^{\ell})$.
\item\label[step]{step:firstComponent} For every $j$, add an edge from $(s^{j}_{n-1}\baseFlag, \ell)$ to $\left(s_{n-1}^{-j}\baseFlag, \ell + (-1)^{\ell}\bar{c}(s_{n-1}^{j}\baseFlag )\right)$.
\end{enumerate}
Observe that the edge of \cref{step:firstComponent} is well defined by \cref{rem:LastEntryDBP}, since the order of $s_{n-1}$ must be even, which implies that $s_{n-1}^{j}\baseFlag$ and $s_{n-1}^{-j}\baseFlag$ have the same colour.
Now for every $\ell \in \bZ_{2s}$ and $k \in \{1, \dots, n-2\}$ let $E_{k}^{\ell}$ denote the $\{k, \dots, n-1\}$-component of $(\baseFlag, \ell)$. 
The vertices of $E_{k}^{\ell}$ are of the form $(\Psi, \ell)$ where $\Psi$ belongs to the orbit of $\baseFlag$ under $ \langle s_{k}, \dots, s_{n-1} \rangle$. 
This implies that $E_{n-2}^{\ell} \subset \cdots \subset E_{1}^{\ell} = G_{\ell}$. 
Define also the families \[\mathcal{C}_{k}^{\ell} = \left\{ \{1, \dots, n-2\}\text{-components}\ F\ \text{of}\ G_{\ell} :  F\cap E_{k}^{\ell}\neq \emptyset\ \text{but}\ F\cap E_{k+1}^{\ell}=\emptyset\right\}.\]
\begin{enumerate}[resume*]
  \item\label[step]{step:baseOfFacets} For $k \leq n-2$ and for every $F \in \mathcal{C}_{k}^{\ell}$ with $\ell$ odd, pick a vertex $(\Phi_{F}, \ell)$ in $E_{k}^{\ell}$ and match it to $(\Phi_{F}, \ell + (-1)^{\ell}\bar{c}(\Phi_{F}))$. 
\end{enumerate}
Since every $\{1, \dots n-2\}$-component of $G_{\ell}$ either has a vertex of the form $\left(s_{n-1}^{j}\baseFlag, \ell  \right)$ for some $j$ or belongs to $\mathcal{C}_{k}^{\ell}$ for some $k$, with \cref{step:baseFlag,step:firstComponent,step:baseOfFacets} we have picked exactly one vertex of each $\{1, \dots, n-2\}$-component of the graphs $G_{1}, \dots, G_{2s}$.

% In \cref{fig:toroids_44_31_ChirExt} we show a possible choice of flags on the graph $ G_{1}= \cay(\cK)$ with $\cK$ the toroid $\{4,4\}_{(3,1)}$. The red flags are those of the form $s_{2}^{k}\baseFlag$ and the green flags are those of the $1$-components of $\mathcal{C}_{1}^{1}$. The flags with darker colours are those flags $\Phi$ that satisfy $\bar{c}(\Phi)=-1$.
% 
% % \begin{figure}%\label{fig:}
% % 	\centering
% % 	\def\svgwidth{0.6\textwidth}
% % 	\input{img/toroids_44_31_ChirExt.pdf_tex}
% % %	\includegraphics[]{}
% % 	\caption{A possible choice of the flags after \cref{step:baseFlag,step:firstComponent,step:baseOfFacets}.}\label{fig:toroids_44_31_ChirExt}
% % \end{figure}
% 
Let $(\Phi, \ell)$ be a vertex and let $F$ be its $\{1, \dots, n-2\}$-component. 
Let $(\Phi_{F}, \ell)$ be the unique vertex of $F$ incident to an edge of $M$. 
Observe that since the action of $\conw(\cK)$ is free on the set of white flags, there exists a unique element $w$ of $\langle s_{1}, \dots, s_{n-2} \rangle$ such that $w\Phi_{F} = \Phi$.
In other words, every vertex of $F$ is of the form $(w\Phi_{F}, \ell)$ for some $w \in \langle s_{1}, \dots, s_{n-2} \rangle$. 

Since $\cK$ has regular facets, \cref{prop:autpAsMonPChir} and \cref{part:chirality} of \cref{thm:chiralGroups} imply that there exist an involutory group automorphism  $\rho$ of $ \langle s_{1}, \dots, s_{n-2} \rangle$ mapping $s_{n-2}$ to $s_{n-2}^{-1}$, $s_{n-3}$ to $s_{n-3}s_{n-2}^{2}$ while fixing $s_{i}$ for $1 \leq i \leq n-4$.  
For $w \in \langle s_{1}, \dots, s_{n-2} \rangle$, let $\bar{w}$ denote $\rho(w)$. 
% To define this automorphism we are thinking the connection elements as the permutations they induce on the set of white flags. 
% Recall that by \cref{prop:autpAsMonPChir} the permutation group $\langle s_{1}, \dots s_{n-2} \rangle$ is isomorphic to the rotation group of the facet of $\cK$. 
Note that the automorphism $\rho$ is the dual version of the automorphism in $\cref{part:chirality}$ of \cref{thm:chiralGroups}. 
\begin{enumerate}[resume*]
  \item\label[step]{step:extendingM} For every vertex of the form $(w\Phi_{F} , \ell)$, with $w \in \langle s_{1}, \dots, s_{n-2} \rangle$, add an edge from $(w\Phi_{F}, \ell)$ to $(\bar{w}\Phi_{F}, \ell+(-1)^{\ell}\bar{c}(\Phi_{F}))$.
\end{enumerate}

Observe that the edge of \cref{step:extendingM} is well-defined since all the flags of the form $w\Phi_{F}$ with $w \in \langle s_{1}, \dots, s_{n-2} \rangle$ have the colour of $\Phi_{F}$. 

We have defined the matching $M$. Let $t$ be the involutory permutation given by swapping the endpoints of each edge of $M$ and define $s_{n} = s_{n-1}^{-1} t$.
% 
% % \begin{figure}%\label{fig:}
% % 	\centering
% % 	\def\svgwidth{\textwidth}
% % 	\input{img/GPR_ChirExtDBP.pdf_tex}
% % %	\includegraphics[]{}
% % 	\caption{The GPR graph of the extension.}
% % \end{figure}
% 
\begin{prop}\label{prop:chiralExtensionDBP}
  Let $\cK$ be a finite dually bipartite chiral polytope with regular facets and let $s \in \bN$. 
  The permutation group $\langle s_{1}, \dots, s_{n} \rangle$ defined by the graph $\gpr$ is the automorphism group of a chiral extension $\cP$ of $\cK$ with the property that $2s$ divides the last entry of the Schläfli symbol of $\cP$.
\end{prop}
\begin{proof}
  We will use \cref{thm:GPRgraphs}. \cref{part:GPR_facets} there follows from our construction, since $G_{1}, \dots, G_{2s}$ are copies of $\cay(\cK)$. To prove \cref{part:GPR_rels} we need to see that the action of $(s_{k} \cdots s_{n})^{2}$ is trivial on every vertex. According to \cref{lem:relsOfExtensionsSigmasAndTau} it suffices to prove that the following relations hold:
  \[\begin{aligned}
	  t^{2} &= \id, \\
    t s_{n-2} t &= s_{n-2}^{-1},\\
    t s_{n-3} t &= s_{n-3}s_{n-2}^{2},\\
    t s_{i}t &= s_{i} && \text{for}\ 1 \leq i \leq n-4.
  \end{aligned}\]
	The first relation holds by construction since $t$ swaps the two vertices of every edge of $M$.
	The other three relations are a consequence of the construction of $M$ in \cref{step:extendingM}. 
	Recall that $\bar{s_{n-2}} = s_{n-2}^{-1}$, $\bar{s_{n-3}} = s_{n-3} s_{n-2}^{2}$, and $\bar{s_{i}} = s_{i}$ for $i \leq n-4$.
	
	To prove \cref{part:GPR_intersection1} of \cref{thm:GPRgraphs} consider the action of $s_{n}$ and $s_{n}^{2}$ on a vertex $(\baseFlag,\ell)$: 
	\begin{equation}\begin{aligned}
	s_{n}\left(\baseFlag, \ell\right) %
  &=s_{n-1}^{-1}t\left(\baseFlag, \ell\right) \\
	&=s_{n-1}^{-1}\left(\baseFlag , \ell+ (-1)^{\ell}\right) \\
	&=\left(s_{n-1}^{-1}\baseFlag, \ell + (-1)^{\ell}\right),
	\end{aligned}\end{equation}
	\begin{equation}\label{eq:DBP_orderOfSigman}\begin{aligned}
		s_{n}^{2}\left(\baseFlag, \ell\right) %
		&= s_{n}\left(s_{n-1}^{-1}\baseFlag, \ell + (-1)^{\ell}\right) \\
		&= s_{n-1}^{-1} t\left(s_{n-1}^{-1}\baseFlag, \ell + (-1)^{\ell}\right) \\
		&= s_{n-1}^{-1} \left(s_{n-1}\baseFlag, \ell + (-1)^{\ell}+(-1)^{\ell+1}\bar{c}(s_{n-1}^{-1}\baseFlag)\right)\\
		&=s_{n-1}^{-1}\left(s_{n-1}\baseFlag, \ell + 2 (-1)^{\ell} \right)\\
		&= \left(\baseFlag, \ell +2 (-1)^{\ell}\right),
	\end{aligned}\end{equation}
	where we have used that $\ell + (-1)^{\ell} \equiv \ell +1 \pmod{2}$ and that $\bar{c}(s^{-1}_{n-1}\baseFlag) = -1$.  
	
	It follows that if $s_{n}^{j}(\baseFlag,0) = (\Psi, 0)$, then $j$ is a multiple of $2s$ and $\Psi = \baseFlag$. 
	In this situation $s_{n}^{j}$ fixes $(\baseFlag,0)$. 
	Since the unique element of $\langle s_{1}, \dots, s_{n-1} \rangle$ fixing $(\baseFlag,0)$ is $\id$ (see \cref{prop:autpAsMonPChir}), $ \langle s_{1}, \dots, s_{n-1} \rangle \cap \langle s_{n} \rangle = \{\id\}$, proving \cref{part:GPR_intersection1}.
	
	Let $k \in \{2, \dots, n-1\}$ and let $D_{k}$ be the $\{k, \dots, n\}$-component of $(\baseFlag, 0)$. 
	In order to prove \cref{part:GPR_intersection2}, we will prove that $D_{k}\cap G_{\ell} = E_{k}^{\ell}$ for every $\ell \in \bZ_{2s}$. 
	It is clear that every vertex of $E_{k}^{\ell}$ belongs to $D_{k} \cap G_{\ell}$, since the elements of $\langle s_{k}, \dots, s_{n-1}\rangle$ map vertices of $G_{\ell}$ to vertices of $G_{\ell}$.  
	To show the other inclusion, we will prove that if $(\Psi, \ell)$ is a vertex in $E_{k}^{\ell}$ then it is matched to a vertex in $E_{k}^{\ell'}$ where $\ell' = \ell + (-1)^{\ell} \bar{c}(\Psi)$ (see \cref{fig:GPR_ChirExtDBP_components}). 
	
\begin{figure}%\label{fig:}
	\centering
	\def\svgwidth{\textwidth}
	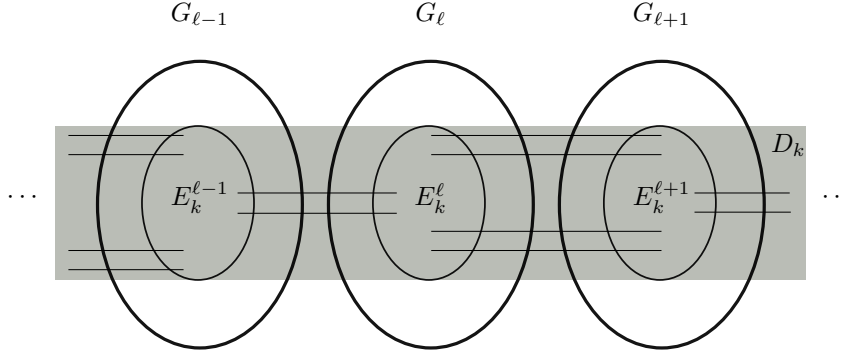
	\caption{The intersection of $D_{k}$ with $G_{\ell}$ is $E_{k}^{\ell}$}\label{fig:GPR_ChirExtDBP_components}
\end{figure}	
Let $(\Psi, \ell)$ be a vertex of $E_{k}^{\ell}$. 
Let $F$ be the $\{1, \dots, n-2\}$-component of $G_{\ell}$ containing $(\Psi, \ell)$. 
Note that $(\Psi, \ell) \in F \cap E_{k}^{\ell}$. Observe that for every vertex $(\Phi, \ell)$ in $F$ we have that $\bar{c}(\Phi) = c(F)$. 
Define $\ell' = \ell + (-1)^{\ell} c(F)$. 
Since $F$ intersects $E_{k}^{\ell}$, then $F \in \mathcal{C}_{j}^{\ell}$ for some $j \geq k$.
In \cref{step:baseOfFacets} of the construction we picked a vertex $(\Phi_{F}, \ell)$ in $E_{j}^{\ell}$ (and hence in $E_{k}^{\ell}$) and matched it to $(\Phi_{F}, \ell')$. 
Since $(\Phi_{F}, \ell)$ and $(\Psi, \ell)$ belong to $F$, there exists $w \in \langle s_{1}, \dots, s_{n-2} \rangle$ such that %
	\[w(\Phi_{F}, \ell) = (\Psi, \ell).\]%
Similarly, since both vertices belong to $E_{k}^{\ell}$, there exists $v \in \langle s_{k}, \dots, s_{n-1} \rangle$ such that%
	  \[v(\Phi_{F}, \ell) = (\Psi, \ell).\]
But since the action of $\langle s_{1}, \dots, s_{n-1} \rangle$ is free (see \cref{prop:autpAsMonPChir}), it follows that $w = v$. 
Therefore, $w$ is an element of $\langle s_{1}, \dots s_{n-2} \rangle \cap \langle s_{k}, \dots, s_{n-1} \rangle = \langle s_{k}, \dots, s_{n-2} \rangle$. 
This implies that $\bar{w} \in \langle s_{k}, \dots, s_{n-2} \rangle$.
	 
Since both $(\baseFlag, \ell)$ and $(\Psi, \ell)$ belong to $E_{k}^{\ell}$, then there exists $u \in \langle s_{k}, \dots, s_{n-1} \rangle$ such that %
	\[(\Psi, \ell) = u(\baseFlag, \ell) .\]%
Finally, in \cref{step:extendingM} we matched $(\Psi, \ell)$ to %
	\[(\bar{w}\Phi_{F}, \ell') = \bar{w}w^{-1}(\Psi, \ell')  = \bar{w}w^{-1}u(\baseFlag, \ell'),\]%
but $\bar{w}w^{-1}u \in \langle s_{k}, \dots, s_{n-1} \rangle$, implying that $(\Psi, \ell)$ is matched to a vertex in $E^{\ell'}_{k}$. 
 	
As a consequence of \cref{thm:GPRgraphs}, the group $\langle s_{1}, \dots, s_{n} \rangle$ is the automorphism group of a chiral extension $\cP$ of $\cK$. To see that the last entry of the Schläfli symbol of $\cP$ must be a multiple of $2s$ just observe that the orbit of $(\baseFlag, 0)$ under $\langle s_{n} \rangle$ has length $2s$ (see \cref{eq:DBP_orderOfSigman}).
\end{proof}

\cref{thm:ManyExtensionsDBP} is an immediate consequence of \cref{prop:chiralExtensionDBP}.

\section{Extensions of chiral polytopes with regular quotients} \label{sec:quotients}

In this section we  describe a technique to build an infinite family $\{\cP_{s} : s \in \bN,\ s \geq 2 \}$ of chiral extensions of a given chiral $n$-polytope $\cK$ with regular facets from a particular chiral extension $\cP$ of $\cK$. 
The polytope $\cP_{s}$ satisfies the property that if $\cP$ has type $\{p_{1}, \dots, p_{n-1}, q\}$, then $\cP_{s}$ is of type $\{p_{1}, \dots, p_{n-1}, \lcm(2s,q)\}$. 
To guarantee the existence of such a family we require the existence of a regular $n$-maniplex $\cR$ with at least two facets such that $\cK$ covers $\cR$.

In this construction we use the maniplex $\twoSM$. 
Given an $n$-maniplex $\cM$, the maniplex $\twoSM$ is a $(n+1)$-maniplex whose facets are isomorphic to $\cM$.
The maniplex $\twoSM$ is finite (resp. regular) if and only if $\cM$ is finite (resp. regular).
Moreover, if every $(n-2)$-face of $\cM$ belongs to two diferent $(n-1)$-faces, then there are $2s$ facets around each $(n-2)$-face of $\twoSM$.
In particular, if $\cM$ is regular with at least two facets, then $\twoSM$ is regular and the last entry of the Schläfli symbol of $\twoSM$ is $2s$.
If $\cK$ is an abstract regular polytope the resulting maniplex $\twoSM[\cK]$ is actually a polytope that is isomorphic to the polytope $\left(\twoSM[\cK^{\ast}]\right)^{\ast}$ introduced by Pellicer in \cite{Pellicer_2009_ExtensionsRegularPolytopes}.
This construction generalises that of Pellicer in the same way as the construction $\hat{2}^{\cM}$ in \cite[Section 6]{DouglasHubardPellicerWilson_2018_TwistOperatorManiplexes} generalises Danzer's $2^{\cK}$ in \cite{Danzer_1984_RegularIncidenceComplexes}.
 
Let $\cM$ be an $n$-maniplex with base flag $\baseFlag$.
Let $\cM_{n-1}$ denote the set of facets of $\cM$ and let $m = |\cM_{n-1}|$.
Recall that since $\cM$ is connected, then $m$ is either finite or countably infinite.
Let $\{F_{j} : 0 \leq j < m \}$ be a labelling of $\cM_{n-1}$ so that  $ \baseFlag \in F_{0}$.
Let $\fl(\cM)$ and $\con(\cM)=\langle r_{0}, \dots, r_{n-1} \rangle$ denote the set of flags and the connection group of $\cM$ respectively. 
Choose $s \in \bN$ such that $s \geq 2$ and let \[A = \bigoplus_{0 \leq j < m}\bZ_{s} = \big\{ \vx = (x_{j})_{0\leq j < m} : x_{j} = 0 \text{ for all but finitely many } j < m \big\}.  \] Let \[U= \left\{\vx \in A : \sum_{0 \leq j < m }x_{j} = 0\right\} \leq A.\]
 
Note that if $m$ is finite, then $A$ is the direct product of $m$ copies of $\bZ_{s}$ and $|U| = s^{m-1}$. 
In most of our applications this is the case.
This motivates our use of vector notation instead of that of formal sequences.
We also abuse of language and call ``vectors'' the elements of $U$ . 

Consider the vectors $\va_{j} = \ve_{j}-\ve_{0}$, where $\ve_{i}$ denotes the vector of $A$ with $i^{th}$ entry equal to $1$ and every other entry equal to $0$. 
Note that $\va_{j} = (-1, \dots, 0, 1, 0 \dots )$ if $0<j$ and that $\va_{0}=\vect{0}$. 
Then \emph{$\twoSM$} is the $(n+1)$-maniplex $\left( \fl_{s}, \{\hat{r}_{0}, \dots, \hat{r}_{n}\}\right)$, where 
\[\begin{aligned}
    \fl_{s} &= \fl(\cM)\times U \times \bZ_{2},\\ %&& \text{with}\ ,\\
    \hat{r}_{i}(\Phi, \vx, \delta) &= (r_{i}\Phi, \vx, \delta) && \text{for}\ 0\leq i \leq n-1, \\
    \hat{r}_{n}(\Phi, \vx, \delta) &= (\Phi, \vx+(-1)^{\delta} \va_{j}, 1-\delta) && \text{whenever}\  \Phi \in F_{j},\\
  \end{aligned}
\]

Now we need to prove that the pair $(\fl_{s}, \{\hat{r}_{0}, \dots, \hat{r}_{n}\})$ indeed defines a maniplex.
First observe that if $i \in \{0, \dots, n\}$, then $\hat{r}_{i}$ is a permutation of $\fl_{s}$, since $r_{i} \Phi \in \fl(\cM)$ and $\va_{i} \in U$. 
Clearly $\hat{r}_{i}$  is an involution for $i \in \{0, \dots, n-1\}$. 
Now assume that $\Phi$ is a flag of $\cM$ with $\Phi  \in F_{j}$, then %
  \[\begin{aligned}
  \hat{r}_{n}^{2}\left(\Phi, \vx, \delta\right) &= \hat{r}_{n}\left(\Phi, \vx + (-1)^{\delta}\va_{j}, 1-\delta \right)\\ 
  &= \left(\Phi, \vx+ (-1)^{\delta}\va_{j} + (-1)^{1-\delta}\va_{j}, \delta\right)\\ 
  &= \left( \Phi, \vx, \delta \right).
  \end{aligned}\]
This proves that $\hat{r}_{n}$ is an involution.
Observe that $r_{i} \mapsto \hat{r}_{i} $ for $0\leq i \leq n-1$ defines an isomorphism between $\con(\cM)$ and $\langle \hat{r}_{0}, \dots, \hat{r}_{n-1} \rangle$. 
This isomorphism proves that the facets of $\twoSM$ are isomorphic to $\cM$. 
Moreover, if $w \in \con(\cM)$ we may think of $w$ as en element of $\con(\twoSM)$. 

The fact that $\con(\twoSM)$ is transitive on $\fl_{s}$ follows from the facts that $\{a_{j} : 0 \leq j < m\}$ is a generating set of $U$ and that $\{r_{0}, \dots, r_{n-1}\}$ is transitive on $\fl(\cM)$. %explained with detail in the following paragraphs, if necesary
%   Let $(\Phi, \vx, \delta)$ be a flag of $\twoSM$. Let $w \in \con(\cM)$ be such that $\Phi=w\baseFlag$. Observe that
%   \begin{equation}\label{eq:twoSM_third}   
%     w\hat{r}_{n}w^{-1}(\Phi, \vx, \delta) = (\Phi, \vx, 1-\delta).
%   \end{equation}
% 	It follows that $( w\hat{r}_{n}w^{-1})^{\delta}(\Phi, \vx, \delta) = (\Phi, \vx, 0)$. 
% 	Now, if $j \in \{1, \dots, m\}$ let $w_{j} \in \con(\cM)$ be such that $w_{j}\Phi \in F_{j}$. Then 
% 	\begin{equation}\label{eq:twoSM_second}
% 		\begin{aligned}
% 	  (w \hat{r}_{n} w^{-1}) (w^{-1}_{j} \hat{r}_{n} w_{j}) (w \hat{r}_{n} w^{-1})^{\delta} (\Phi, \vx, \delta)&= (w \hat{r}_{n} w^{-1}) (w^{-1}_{j} \hat{r}_{n} w_{j}) (\Phi, \vx, 0)\\
% 	  &= (w \hat{r}_{n} w^{-1})(\Phi, \vx+\va_{j},1)\\
% 	  &=(\Phi, \vx+\va_{j},0).
% 	  \end{aligned}
% 	\end{equation}
% 	
% 	Combining \cref{eq:twoSM_third}, \cref{eq:twoSM_second} and the fact that $\{\va_1, \dots, \va_{m}\}$ is a generating set for $U$ it follows that we can map every flag $(\Phi, \vx, \delta)$ to every flag $(\Phi, \vy, \epsilon)$ by an element of $\langle \hat{r}_{0}, \dots, \hat{r}_{n} \rangle$. The transitivity of $\con(\twoSM)$ on $\fl_{s}$ now follows from the transitivity of $\con(\cM)$ on $\fl(\cM)$.
%
Note that for ever $i \in \left\{ 0, \dots, n \right\} $, the permutation $\hat{r}_{i}$ is fixed-point-free.
If $i,j \in \{0, \dots, n-1\}$ it is clear that $\hat{r}_{i}(\Phi, \vx, \delta)\neq \hat{r}_{j}(\Phi,\vx, \delta)$ since $\cM$ is a maniplex. 
For $i \in \{0, \dots, n-1\}$, $\hat{r}_{n} (\Phi, \vx, \delta) \neq \hat{r}_{i}(\Phi, \vx, \delta)$ since they have different third coordinates. 
% This proves that $\twoSM$ is a maniplex.

Let $(\Phi, \vx, \delta)$ be a flag of and let $0 \leq j < m$ be such that $ \Phi \in F_{j}$. 
If  $i \in \{0, \dots, n-2\}$ then $ r_{i}\Phi \in F_{j}$, and then it follows that 
\[\hat{r}_{i}\hat{r}_{n}(\Phi, \vx, \delta) = (r_{i} \Phi, \vx + (-1)^{\delta}\va_{j}, 1-\delta) = \hat{r}_{n}\hat{r}_{i}(\Phi,\vx, \delta).\] 
Then $\hat{r}_{n}$ commutes with $\hat{r}_{i}$ whenever ${i} \leq n-2$. 
The elements $\hat{r}_{i}$ and $\hat{r}_{j}$ commute for $i,j \in \{0, \dots, n-1\}$ with $i \neq j$ and $|i-j|\geq 2$ because $r_{i}$ and $r_{j}$ commute. 
We have proved then that $\twoSM$ is a maniplex.
	
Let $(\Phi, \vx, \delta)$ be a flag and assume that $ \Phi \in F_{j}$. 
Let $0 \leq k < m $ be such that $r_{n-1}(\Phi) \in F_{k}$. 
Observe that if every $(n-2)$-face of $\cM$ is incident to two $(n-1)$-faces, then $j \neq k$. 
Consider the action of $(\hat{r}_{n} \hat{r}_{n-1})^{2}$ on an arbitrary flag $(\Phi, \vx, \delta)$:
\[\begin{aligned}
 (\hat{r}_{n}\hat{r}_{n-1})^{2}\left( \Phi, \vx, \delta \right) %
 &= \hat{r}_{n}\hat{r}_{n-1}\hat{r}_{n} \left(r_{n-1}\Phi, \vx ,\delta \right) \\
 &= \hat{r}_{n} \hat{r}_{n-1} \left(r_{n-1} \Phi, \vx +(-1)^{\delta}\va_{k}, 1-\delta \right)\\
 &= \hat{r}_{n} \left(\Phi, \vx +(-1)^{\delta}\va_{k}, 1-\delta \right)\\
 &= \left( \Phi, \vx + (-1)^{\delta} \va_{k} + (-1)^{1-\delta} \va_{j}, \delta\right)\\
 &=\left(\Phi, \vx + (-1)^{\delta} (\va_{k} - \va_{j}), \delta \right).
\end{aligned}\] 
It follows that orbit of $(\Phi, \vx, \delta)$ under  $\langle (\hat{r}_{n} \hat{r}_{n-1})^{2} \rangle$ has length the same as the order of $(\va_{k} - \va_{j})$	in $U$, which is $s$, provided that $j \neq k$. 
Therefore, if every $(n-2)$-face of $\cM$ is incident to two $(n-1)$-faces, then the order of $(\hat{r}_{n} \hat{r}_{n-1})^{2}$ is $s$. 
Then we have that there are $2s$ facets incident to each $(n-2)$-face of $\twoSM$. 
In particular, if $\cM$ is of type $\{p_{1}, \dots, p_{n-1}\}$, then $\twoSM$ is of type $\{p_{1}, \dots, p_{n-1}, 2s\}$.			                                
If $\cM$ is a regular maniplex, then the condition that every $(n-2)$-face of $\cM$ is incident to two facets of $\cM$ is equivalent to requiring that $\cM$ has at least two facets.
	
Observe that the orbit of a flag $\left( \Phi, \vx, \delta \right)$ under the group $\langle \hat{r}_{0}, \dots, \hat{r}_{n-1} \rangle$ is \[\{(\Psi, \vx, \delta) : \Psi \in \fl(\cM)\}.\] 
In particular, every facet of $\twoSM$ is determined by a pair $(\vx, \delta)$ with $\vx \in U$ and $\delta \in \bZ_{2}$.
Furthermore, the maniplex $\twoSM$ is dually bipartite; and the necessary colouring is given by $\delta$. 
The \emph{base facet} of $\twoSM$ is the facet determined by the pair $(\vect{0},0)$. 
This facet together with the base flag $\baseFlag$ of $\cM$ determines the base flag $(\baseFlag, \vect{0}, 0)$ of $\twoSM$. 

% In \cref{fig:twoSM_triang} we show then maniplex $\twoSM$ when $\cM$ is a triangle and $s=3$. The facets of the triangle are labelled with $\{1,2,3\}$ being $3$ the base facet. The shaded triangles are those whose flags satisfy $\delta = 1$ and the white triangles are those with $\delta = 0$. Each facet has associated to it an element of $U$ as explained in the previous paragraph.
% 
% \begin{figure}%\label{fig:}
% 	\centering
% 	\def\svgwidth{\textwidth}
% % 	\begin{scriptsize}
% 	\input{img/twoSM_triang.pdf_tex}
% % 	\end{scriptsize}
% %	\includegraphics[]{}
% 	\caption{The maniplex $\twoSM$ where $\cM$ is a triangle and $s=3$.}\label{fig:twoSM_triang}
% \end{figure}
% 
Now we describe some symmetry properties of $\twoSM$. 
Assume that $\gamma \in \aut(\cM)$. 
Observe that $\gamma$ acts on $\{j : 0 \leq j < m\}$ by permuting the elements the same way it permutes the facets of $\cM$. 
If $\vx =(x_{j} : j < m)$, then take $\vx \gamma := (x_{j \gamma^{-1}}: j < m)$. 
This defines a right action of $\aut(\cM)$ on $U$. 
Observe that this action is linear in the sense that $(\vx + \vy)\gamma = \vx \gamma+ \vy \gamma$ for every $\vx, \vy \in U$. 
Note also that $\ve_{j}\gamma = \ve_{j\gamma}$ for every $j \in \{1, \dots, m\}$.  
In particular $\va_{j} \gamma = \ve_{j} \gamma - \ve_{0} \gamma = \ve_{j\gamma} - \ve_{0 \gamma}$. 

For $\gamma \in \aut(\cM)$, define $\bar{\gamma}:\fl_{s} \to \fl_{s}$ by \[(\Phi, \vx, \delta) \bar{\gamma} = (\Phi \gamma, \vx \gamma +\delta \va_{0\gamma}, \delta).\] 
\begin{prop}\label{prop:twoSM_StabBaseFacet}
The mapping $\gamma \mapsto \bar{\gamma}$ defines an embedding of $\aut(\cM)$ into $\aut(\twoSM)$. 
In fact, if $\bar{\cM}$ denotes the base facet of $\twoSM$, then the image of $\aut(\cM)$ is precisely $\stab_{\aut(\twoSM)}(\bar{\cM})$.
\end{prop}
\begin{proof}
In order to see that, we first need to prove that $\bar{\gamma}$ actually defines an automorphism. 
Note that $\bar{\gamma}$ is a permutation of $\fl_{s}$ with inverse $\bar{\gamma^{-1}}$. 
%   \[\begin{aligned}
% 	  \left(\Phi\gamma, \vx \gamma + \delta \va_{m \gamma} , \delta \right)\bar{\gamma^{-1}}%
% 	  &= \left( \Phi, (\vx \gamma + \delta \va_{m\gamma})\gamma^{-1}+ \delta \va_{m \gamma^{-1}}, \delta \right)\\
% 	  &=\left(\Phi, \vx + \delta (\ve_{m \gamma \gamma^{-1}} - \ve_{m \gamma^{-1}}) + \delta(\ve_{m \gamma^{-1}}- \ve_{m}), \delta\right)\\
% 	  &= \left(\Phi, \vx, \delta \right).
%   \end{aligned}\] 

Now we need to prove that 
\[ \hat{r}_{i} \left( (\Phi, \vect{x}, \delta) \bar{\gamma}  \right) = \left( \hat{r}_{i}(\Phi, \vect{x}, \delta) \right)\bar{\gamma} \] 
for every $i \in \left\{ 0, \dots, n \right\}$ and every $(\Phi, \vect{x}, \delta) \in \fl_{s}$.
 
If $i \in \{1, \dots, n-1\}$, then %
\[\begin{aligned}
	  \hat{r}_{i}\left( \left( \Phi, \vx, \delta\right)\bar{\gamma}\right) %
	  &= \hat{r}_{i}\left(\Phi\gamma, \vx\gamma+\delta\va_{0\gamma}, \delta\right) \\
	  &= \left(r_{i} \Phi \gamma, \vx \gamma + \delta \va_{0 \gamma}, \delta \right)\\
	  &= \left(r_{i} \Phi, \vx, \delta\right) \bar{\gamma} \\
	  &= \left(\hat{r}_{i}\left( \Phi, \vx, \delta \right)\right)\bar{\gamma}.
  \end{aligned}\]
So it remains to show that $\hat{r}_{n}\left( \left( \Phi, \vx, \delta \right)\bar{\gamma} \right)  = \left( \hat{r}_{n}\left(  \Phi, \vx, \delta\right) \right)\bar{\gamma}$.

Consider the left side of the previous equation:
\[\begin{aligned}
	\hat{r}_{n}\left( \left( \Phi, \vx, \delta \right)\bar{\gamma} \right)%
	&= \hat{r}_{n}\left( \Phi \gamma, \vx \gamma + \delta \va_{0\gamma}, \delta \right)\\
	&=\left( \Phi \gamma, \vx \gamma + \delta \va_{0\gamma} + (-1)^{\delta}\va_{j}, 1-\delta \right),
\end{aligned}\] 
 where $j$ is such that $ \Phi \gamma \in F_{j} $. 
 Now the right side:
\[\begin{aligned}
  \left( \hat{r}_{n}\left( \Phi, \vx, \delta \right) \right)\bar{\gamma} %
  &=\left( \Phi, \vx + (-1)^{\delta}\va_{k}, 1-\delta \right)\bar{\gamma}\\
  &=\left( \Phi \gamma, (\vx + (-1)^{\delta} \va_{k})\gamma + (1-\delta)\va_{0\gamma}, 1-\delta \right)
\end{aligned}\]
where $k$ is such that $\Phi \in F_{k}$. 
Note that if these two are different, then they differ in the second coordinate. 
Observe also that $F_{k} \gamma = F_{j}$.
 
Compare the second coordinates. 
On the one hand,%
\[\begin{aligned}
	\vx \gamma + \delta \va_{0\gamma} + (-1)^{\delta}\va_{j}%
	&= \vx \gamma + \delta(\ve_{0\gamma}-\ve_{0}) + (-1)^{\delta}(\ve_{j}-\ve_{0})\\
	&=\vx \gamma + \delta \ve_{0\gamma} + \left(-\delta - (-1)^{\delta}\right)\ve_{0} + (-1)^{\delta}\ve_{j},
\end{aligned}\]%
and on the other%
\[\begin{aligned}
	(\vx + (-1)^{\delta} \va_{k})\gamma + (1-\delta)\va_{0\gamma}%
	&= \vx \gamma + (-1)^{\delta} (\ve_{k\gamma} - \ve_{0\gamma})+(1-\delta)(\ve_{0\gamma}-\ve_{0}) \\
	&= \vx \gamma + \left(- (-1)^{\delta} + (1-\delta)\right)e_{0\gamma} - (1-\delta)\ve_{0} + (-1)^{\delta}\ve_{j}.
\end{aligned}\]
Finally, observe that $\left( - (-1)^{\delta} + (1-\delta) \right) = \delta$ and $\left(-\delta - (-1)^{\delta}\right)= (\delta-1)$ for $\delta \in \left\{ 0,1 \right\}$. 
This proves that $\bar{\gamma}$ is indeed an automorphism of $\twoSM$. 

Clearly if $(\Phi, \vect{0}, 0)$ is a flag of the base facet, then \[(\Phi, \vect{0},0)\bar{\gamma} = (\Phi\gamma, \vect{0},0).\] 
Since $\Phi$ is arbitrary, this defines a group homomorphism such that the image of $\aut(\cM)$ is a subgroup of $\stab_{\aut(\twoSM)}(\bar{\cM})$. 
Since the action of the automorphism group on the set of flags is free, this homomorphism is injective.   
Every element of $\stab_{\aut(\twoSM)}(\bar{\cM})$ induces an automorphism of $\bar{\cM}$.
Since the action of $\aut(\twoSM)$ on flags is free, these automorphisms must belong to the image of $\aut(\cM)$.
\end{proof}

From now on we will abuse notation and denote $\bar{\gamma}$ simply by $\gamma$ and think of $\aut(\cM)$ as a subgroup of $\aut(\twoSM)$. 
Similarly, if there is no place for confusion, we will denote the base facet of $\twoSM$ by $\cM$ instead of $\bar{\cM}$.

Now, for every $\vy \in U$ consider $\tau_{\vy}: \fl_{s} \to \fl_{s}$ given by $\tau_{\vy}: (\Phi, \vx, \delta) \mapsto (\Phi, \vx+ \vy, \delta)$. 
Consider also the mapping $\chi: \fl_{s} \to \fl_{s}$ given by $\chi: (\Phi, \vx, \delta) \mapsto (\Phi, - \vx, 1-\delta)$. 
Straighforward computations show that $\tau_{\vy}$ and $\chi$ define automorphisms of $\twoSM$.
In fact, it can be proved that \[\left\langle  \left\{ \chi \right\}  \cup  \left\{ \tau_{\vy} : \vy  \in U \right\}  \right\rangle =  \left\langle \chi \right\rangle \ltimes \left\langle  \tau_{\vy} : \vy  \in U  \right\rangle  \cong \bZ_{2} \ltimes U .  \]
Moreover, we can fully describe the structure of $\aut(\twoSM)$.

\begin{thm}\label{thm:autTwoSM}
  Let $\cM$ be an $n$-maniplex such that every $(n-2)$-face of $\cM$ is incident to two facets. Then \[\aut(\twoSM) \cong \aut(\cM)\ltimes \left( \bZ_{2} \ltimes U \right). \]
\end{thm}
\begin{proof}
  The discussion above implies that \[\left\langle \aut(\cM) \cup \left\{ \chi \right\} \cup \left\{ \tau_{\vy}: \vy\in U \right\} \right\rangle \leq \aut(\twoSM).\] 
  To prove the other inclusion observe that if an automorphism $\omega \in \aut(\twoSM)$ maps the base flag $\left( \baseFlag, \vect{0}, 0 \right)$ to $(\Psi, \vx, \delta)$, there is an automorphism of $\bZ_{2} \ltimes U$ mapping $(\baseFlag, \vect{0},0)$ to $(\Psi, \vect{0}, 0)$, and then there must be an automorphism in $\aut(\cM) $ mapping $\Psi_{0}$ to  $\Phi$, since $\aut(\cM)$ is the stabilizer of the base facet. 
  The inclusion follows from the fact the the action of $\aut(\twoSM)$ on $\fl_{s}$ is free.
  
  It just remains to determine the structure of the group. It is clear that $\left( \bZ_{2} \ltimes U \right) \cap \aut(\cM) = \{\id\}$, since the former fixes the first coordinate of every flag and the only element of $\aut(\cM)$ that fixes a flag of $\cM$ is $\id$. Take $\gamma \in \aut(\cM)$, $\vy \in U$ and let $(\Phi, \vx, \delta)$ be an arbitrary flag. 
  Then
  \[ \begin{aligned}
     	\left( \baseFlag, \vx, 0 \right) \gamma^{-1} \tau_{\vy} \gamma% 
     	&= \left(\Phi \gamma^{-1}, \vx \gamma^{-1} + \delta \va_{0 \gamma^{-1}}, \delta\right) \tau_{\vy} \gamma\\
     	&= \left(\Phi \gamma^{-1}, \vx \gamma^{-1} + \delta \va_{0 \gamma^{-1}} + \vy, \delta\right)\gamma\\
     	&= \left( \Phi, (\vx\gamma^{-1} + \delta \va_{0 \gamma^{-1}} + \vy) \gamma + \delta \va_{0\gamma}, \delta \right)\\
     	&=\left( \Phi, \vx + \delta(\ve_{0\gamma^{-1} \gamma}-\ve_{0\gamma} + \ve_{0 \gamma} - \ve_{0}) + \vy \gamma, \delta \right)\\
     	&=\left( \Phi, \vx + \vy \gamma, \delta \right)\\
     	&= \left( \Phi, \vx, \delta \right)\tau_{\vy\gamma}.
   \end{aligned}\]
 Similarly,
		\[\begin{aligned}
			 \left( \Phi, \vx, \delta \right)\gamma^{-1} \chi \gamma%
			 &=\left( \Phi \gamma^{-1}, \vx\gamma^{-1} + \delta \va_{0 \gamma^{-1}}, \delta\right) \chi \gamma \\
			 &=\left(\Phi \gamma^{-1}, -(\vx\gamma^{-1} + \delta \va_{0 \gamma^{-1}}), 1-\delta\right) \gamma\\
			 &= \left( \Phi, -\vx + \delta(\ve_{0} - \ve_{0\gamma^{-1}} ) \gamma + (1-\delta)(\ve_{0\gamma} - \ve_{0}), 1-\delta \right)\\
			 &=\left( \Phi, -\vx + \va_{0\gamma}, 1- \delta \right)\\
			 &= \left( \Phi, \vx, \delta \right) \chi \tau_{\va_{0\gamma}}. 
		\end{aligned} \]
%   These computations imply that $\gamma^{-1} \tau_{\vy} \gamma = \tau_{\vy \gamma}$ and $\gamma^{-1} \chi \gamma = \chi \tau_{\va_{m\gamma}}$. 
Therefore, $ \aut(\cM)$ normalises $(U \ltimes \bZ_{2}) $  and \[\left\langle \aut(\cM) \cup \left\{ \chi \right\} \cup \left\{ \tau_{\vy}: \vy\in U \right\}  \right\rangle = \aut(\cM)\ltimes \left( \bZ_{2} \ltimes U  \right). \qedhere \]
\end{proof} 
% 
% \begin{coro}\label{coro:TwoSMregular}
%  Let $\cM$ be a regular $n$-maniplex with at least two facets and such that $\aut(\cK) = \langle \rho_{0}, \dots, \rho_{n-1} \rangle$. Then $\twoSM$ is a regular maniplex and $\rho_{0}, \dots, \rho_{n-1}, \chi$ act as abstract reflections with respect to the base flag $(\baseFlag, \vect{0}, 0)$.
% \end{coro}
% \begin{proof}
%   Just observe that for $i \in \left\{ 0, \dots, n-1 \right\}$ \[\hat{r}_{i}\left( \baseFlag, \vect{0} ,0 \right) = \left( r_{i}\Phi, \vect{0},0 \right) = \left( \Phi \rho_{i}, \vect{0},0 \right) = \left( \Phi, \vect{0},0 \right) \rho_{i}\] and \[\hat{r}_{n}\left( \baseFlag, \vect{0} ,0 \right) = \left( \Phi, \vect{0}, 1 \right) =  \left( \Phi, \vect{0},0 \right) \chi. \qedhere\] 
% \end{proof}
% 
% \begin{coro}\label{coro:IsomTwoSK}
%   If $\cK$ is a regular polytope, then the maniplex $\twoSM[\cK]$ is an abstract regular polytope and it is isomorphic to the polytope $\left( \twoSK[\cK^{\ast}] \right)^{\ast}$ constructed by Pellicer in  \emph{\cite{pellicer_ExtensiosOfARPWithPreassignedST}}. 
% \end{coro}
% \begin{proof}
%   The automorphism group described in \cref{thm:autTwoSM} is isomorphic to the one described in \cite[Theorem 3.4]{pellicer_ExtensiosOfARPWithPreassignedST}. The isomorphism maps each abstract reflection of $\twoSM[\cK]$ to the corresponding abstract reflection of $\left( \twoSK[\cK^{\ast}] \right)^{\ast}$.
% \end{proof}

Observe that if $\cM$ is a regular $n$-maniplex with automorphism group $\left\langle \rho_{0}, \dots, \rho_{n-1} \right\rangle $,  then $\twoSM$ is a regular $(n+1)$-maniplex and $\rho_{0}, \dots, \rho_{n-1}, \chi$ act as generating reflections of $\aut(\twoSM)$. 

If $\cM$ is a finite regular polytope, then $\twoSM$ is the polytope $\left( 2s^{\cM^{\ast} -1} \right)^{\ast}$, where $2s^{\cM^{\ast}-1}$ is the construction introduced by Pellicer in \cite{Pellicer_2009_ExtensionsRegularPolytopes} applied to the dual of $\cM$.
To see this, just observe that the automorphism group described in \cref{thm:autTwoSM} is isomorphic to the one described in \cite[Theorem 3.4]{Pellicer_2009_ExtensionsRegularPolytopes}. 
The isomorphism maps each abstract reflection of $\twoSM$ to the corresponding abstract reflection of $\left( 2s^{\cM^{\ast} -1} \right)^{\ast}$.
In this sense, the construction introduced here slightly generalises that of \cite{Pellicer_2009_ExtensionsRegularPolytopes} since the former can be applied to non-polytopal and/or non-regular maniplexes.

The construction $\twoSM$ allows us to prove the following theorem. 
\begin{thm}\label{thm:ExtensionsRegQuotients}
  Let $\cK$ be a chiral $n$-polytope of type $\{p_{1}, \dots, p_{n-1}\}$. Assume that $\cK$ has a quotient that is a regular maniplex with at least two facets. If $\cP$ is a chiral extension of $\cK$ of type $\{p_{1}, \dots, p_{n-1}, q\}$, then for every $s \in \bN$, $\cK$ has a chiral extension of type $\{p_{1}, \dots, p_{n-1}, \lcm(q,2s)\}$.
\end{thm}
\begin{proof}
  Let $\autp(\cK) = \langle \sigma_{1}, \dots, \sigma_{n-1} \rangle$ and $\autp(\cP) = \langle \sigma_{1}, \dots, \sigma_{n-1}, \sigma_{n} \rangle$. 
  Let $\cR$ be the regular quotient of $\cK$. 
  For $s \in \bN$, consider $\twoSM[\cR]$. 
  By the discussion above, $\twoSM[\cR]$ is a regular $(n+1)$-maniplex. 
  Let $\autp(\twoSM[\cR]) = \langle \sigma'_{1}, \dots, \sigma'_{n} \rangle$. 
  Observe that $\autp(\cR) = \langle \sigma'_{1}, \dots, \sigma'_{n-1} \rangle$. 
  Consider the group \[\Gamma_{s} = \autp(\cP) \mix \autp(\twoSM[\cR]) = \left\langle (\sigma_{1}, \sigma'_{1}), \dots, (\sigma_{n}, \sigma'_{n})  \right\rangle. \]   
  
  Observe that the group $\autp(\cK)=\langle \sigma_{1}, \dots, \sigma_{n-1} \rangle$ covers the group $\autp(\cR) = \langle \sigma'_{1}, \dots, \sigma'_{n-1} \rangle $. Since $\autp(\cP)$ satisfies the intersection property, then \cref{lem:quotientCriterion} implies that $\Gamma_{s}$ is the automorphism group of a rotary polytope $\cP_{s}$. 
  By \cref{lem:covers}, the group \[\autp(\cK) \mix \autp(\cR) = \left\langle (\sigma_{1}, \sigma'_{1}), \dots, (\sigma_{n-1}, \sigma'_{n-1})  \right\rangle \cong \autp(\cK).\] 
  It follows that the facets of $\cP_{s}$ are isomorphic to $\cK$ and hence they are chiral. 
  Therefore $\cP_{s}$ is chiral itself. 
  The order of $(\sigma_{n}, \sigma'_{n})$ is $\lcm(q,2s)$, which implies that $\cP_{s}$ is of type $\{p_{1}, \dots, p_{n-1}, \lcm(q,2s)\}$.
\end{proof}

\section{Chiral extensios of maps on the torus}\label{sec:examples}
In this section we show some applications of the constructions developed in \cref{sec:dually,sec:quotients}. 

\begin{exam}\label{exam:maps44_DB}
 If $b,c \in \bZ$ such that $bc(b-c) \neq 0$ and $b \equiv c \pmod{2}$, then the map $\left\{ 4,4 \right\}_{(b,c)} $ is a dually-bipartite chiral $3$-polytope.
 By applying the construction in \cref{sec:dually} we obtain an family of chiral $4$-polytopes of type $\left\{ 4,4, 2q \right\} $ for infinitely many $q \in \bN$.
\end{exam}

\begin{exam}\label{exam:maps36_DB}
 Similarly if $bc(b-c) \neq 0$ the map $\left\{ 3,6 \right\}_{(b,c)}$ is a dually-bipartite chiral $3$-polytope. 
 The construction in \cref{sec:dually} gives rise to a family of $4$-polytopes of type $\{3,6,2q\}$ for infinitely many values of $q \in \bN$.
\end{exam}

The condition for $\cK$ to be dually-bipartite can be slightly restrictive. 
For example the maps $\left\{ 6,3 \right\}_{(b,c)}$ are never dually-bipartite.
However the construction outlined in \cref{sec:quotients} can be applied to many of those maps, as it will be shown in the discussion below.

\begin{exam}\label{exam:maps44_42}
	The map $\cK=\{4,4\}_{(4,2)}$ has infinitely many non-isomorphic chiral extensions. 
	Moreover, $\cK$ admits a chiral extension whose last entry of its Schläfli symbol is arbitrarily large.
Indeed, by \cite[Theorem 1]{CunninghamPellicer_2014_ChiralExtensionsChiral}, $\cK$ admits a chiral extension $\cP$. 
Let $\bLL_{(b,c)}$ denote the lattice group such that $\{4,4\}_{(b,c)} = \{4,4\}/\bLL_{(b,c)}$. 
Since the lattice group $\bLL_{(4,2)}$ is contained in the lattice $\bLL_{(2,0)}$, then $\cR = \{4,4\}_{(2,0)}$ is a regular quotient of $\cK$ with $4$ facets. 
Therefore we may apply \cref{thm:ExtensionsRegQuotients}. 
\end{exam}

Of course the idea in \cref{exam:maps44_42} extends to a map $\left\{ 4,4 \right\}_{(b,c)} $ as long as such map has a regular quotient with at least two facets. 
The results in \cite[Section 4]{BredaDAzevedoNedela_2006_ChiralityGroupChirality} imply that the only maps on the torus without a regular quotient with at least two facets are $\{4,4\}_{(b,c)}$ and $\{6,3\}_{(b,c)}$ with $b,c$ coprime and $b \not\equiv c \pmod{2}$ for type $\{4,4\}$, or $b \not\equiv c \pmod{3}$ for the case $\{6,3\}$. 
Note that every regular toroidal map of type $\{3,6\}$ has at least two facets. 
Therefore, \cref{exam:maps44_42} extends to such maps.
In the discusion below we will give explicit Schläfli symbols for chiral extensions of such maps by using a slightly different approach.

% Using these facts and the discussion about chiral extensions of maps above together with \cref{thm:ExtensionsRegQuotients} we have proved the following results.

In \cite{SchulteWeiss_1994_ChiralityProjectiveLinear} Schulte and Weiss build chiral $4$-polytopes from the hyperbolic tessellations of types $\left\{ 4,4,3 \right\}$, $\left\{ 4,4,4 \right\}$, $\left\{  6,3,3\right\}$ and $\left\{ 3,6,3 \right\}$. 
The strategy in all the cases is essentially the same: they represent the rotation group $[p,q,r]^{+}$ of $\left\{ p,q,r \right\}$ as matrices in $PGL_{2}(R)$  with $R=\bZ[i]$ (the Gaussian integers) if $\left\{ p,q,r \right\} $ is $\left\{ 4,4,4 \right\}$ or $\left\{ 4,4,3 \right\}$, and $ R= \bZ[\omega]$ (the Eisenstein integers) if $\left\{ p,q,r \right\}$ is $\left\{ 6,3,3 \right\}$ or $\left\{ 3,6,3 \right\}$. 
Then they find an appropriate $m$ in such a way there is a ring homomorphism from $R$ to $\bZ_{m}$. This induces a homomorphism from $PGL_{2}(R)$ to $ PGL_{2}(\bZ_{m})$ which maps $[p,q,r]^{+}$ to a finite group. Then they prove that these finite groups are the rotation groups of finite orientably regular or chiral polytopes. 

The following results are simplified versions of some of their results.

\begin{thm}[{\cite[Theorem 7.3]{SchulteWeiss_1994_ChiralityProjectiveLinear}}]\label{thm:443reg}
  For every integer $m \geq 3$ there exists an orientably regular polytope $\cP$ of type $\{4,4,3\}$ with facets isomorphic to $\{4,4\}_{(m,0)}$ and vertex-figures isomorphic to the cube $\{4,3\}$.
\end{thm}

\begin{thm}[{\cite[Theorem 7.6]{SchulteWeiss_1994_ChiralityProjectiveLinear}}]\label{thm:443chir}
  Let $m \geq 3$ be an integer such that the equation $x^{2}+1$ has a solution in $\bZ_{m}$.  Then for every $i \in \bZ_{m}$ such that $i^{2} \equiv -1 \pmod{m}$ there exist $b,c \in \bZ$ such that $\gcd(b,c)=1$, $m = b^{2}+c^{2}$ and $b+ci \equiv 0   \pmod{m}$. In this situation, the image of $[4,4,3]^{+}$ to $PGL_{2}(\bZ_{m})$ is the automorphism group of a chiral polytope of type $\{4,4,3\}$ with facets isomorphic to $\{4,4\}_{(b,c)}$ and vertex-figures isomorphic to $\{4,3\}$.
\end{thm}

Note that if we start with arbitrary $b,c \in \bZ$ such that $\gcd(b,c) = 1$ then $m=b^{2}+c^{2}$ satisfies the hypotesis of \cref{thm:443chir}. 
It follows that we can build a chiral extension of type $\{4,4,3\}$ for every map $\left\{ 4,4 \right\}_{(b,c)} $ with $\gcd(b,c) = 1$.
It can be proved that for every pair of integers $b,c$ there exist $b_{1}, \dots, b_{k}, c_{1}, \dots, c_{k} \in \bZ$ such that 
\[\autp\left(\left\{ 4,4 \right\}_{(b,c)}\right) \cong \autp\left(\left\{4,4\right\}_{(b_{1},c_{1})}\right) \mix \cdots \mix \autp\left(  \left\{ 4,4 \right\}_{(b_{k},c_{k})}\right),\] 
and such that if $b_{j}c_{j} \neq 0$, then $b_{j}$ and $c_{j}$ are coprime.
For each $j \in \{1, \dots, k\}$ let $\cP_{j}$ be the extension of $\{4,4\}_{(b_{j},c_{j})}$ given by \cref{thm:443reg} or \cref{thm:443chir}. 
Let \[\Gamma = \autp\left( \cP_{1}\right) \mix \cdots \mix \autp\left( \cP_{k} \right). \]
The group $\Gamma$ satisfy the relations in \cref{eq:relsSigmas}.
The intersection propety for $\Gamma$ follows from a dual version of \cref{lem:quotientCriterion} since all the vertex-figures of the polytopes $\cP_{1}, \dots, \cP_{k}$ are isomorphic to the cube.
Therefore the polytope $\cP=\cP(\Gamma)$ obtained from \cref{thm:chiralGroups} is a rotary polytope of type $\{4,4,3\}$. 
Observe that the facets of $\cP$ are precisely $\left\{ 4,4 \right\}_{(b,c)}$. This implies that $\cP$ is a chiral extension of  $\left\{ 4,4 \right\}_{(b,c)}$ of type $\left\{ 4,4,3 \right\} $.

Of course, the analysis we have done for type $\left\{ 4,4,3 \right\}$ can be done for types $\left\{ 4,4,4 \right\}$, $\left\{ 6,3,3 \right\}$ and $\left\{ 3,6,3 \right\}$ (using the results of Section 8, Section 9 and Section 10 of \cite{SchulteWeiss_1994_ChiralityProjectiveLinear}, respectively), and we  obtain chiral extensions of the maps $\{4,4\}_{(b,c)}$, $\left\{ 3,6 \right\}_{(b,c)}$ and $\{6,3\}_{(b,c)}$. 
With those polytopes of type $\{6,3,3\}$ the intersection property will follow again from \cref{lem:quotientCriterion}. 
For types $\{4,4,4\}$ and $\{3,6,3\}$ the intersection property involve long but straightforward computations using the matices associated to the rotation groups.
The proof can be done following the same idea as in \cite[Sections 8 and 10]{SchulteWeiss_1994_ChiralityProjectiveLinear}. 

By combining the discussion above with \cref{thm:ExtensionsRegQuotients}, we can prove the following results. 

\begin{thm}\label{thm:chiralExtOfMaps444}
	Let $\cM=\{4,4\}_{(b,c)}$ be a chiral map of type $\{4,4\}$ on the torus. Let $d=\gcd(b,c)$. If $d \neq 1$ or $d=1$ but $b \equiv c \pmod{2}$, then for every $s \in \bN$ there exists a chiral extension of $\cM$ of type $\{4,4,6s\}$ and a chiral extension of $\cM$ of type $\{4,4,4s\}$.
\end{thm}

\begin{thm}\label{thm:chiralExtOfMaps366}
	Let $\cM=\{3,6\}_{(b,c)}$ be a chiral map of type $\{3,6\}$ on the torus. Then for every $s \in \bN$ there exists a chiral extension of $\cM$ of type $\{3,6,6s\}$.
\end{thm}

\begin{thm}\label{thm:chiralExtOfMaps636}
	Let $\cM=\{6,3\}_{(b,c)}$ be a chiral map of type $\{6,3\}$ on the torus. Let $d=\gcd(b,c)$. If $d \neq 1$ or $d=1$ but $b \equiv c \pmod{3}$, then for every $s \in \bN$ there exists a chiral extension of $\cM$ of type $\{6,3,6s\}$.
\end{thm}

\section{Acknowledgements}
The author wishes to thank Daniel Pellicer for all the patience and the several hours of duscussion about the results developed in this paper. He also wishes to thank Gabe Cunningham for helping to understand the mathematics required for the last section.

\bibliographystyle{plainurl}
\bibliography{chirExtSchlafli.bib}

\end{document}